\documentclass[12pt]{amsart}
\usepackage{}
\usepackage{amsmath}
\usepackage{amsfonts}
\usepackage{amssymb}
\usepackage[all]{xy}           %xypic macro for latex2.09
\usepackage{bm}
\usepackage{bbding}
\usepackage{txfonts}
\usepackage{amscd}
\usepackage{xspace}
\usepackage[shortlabels]{enumitem}
\usepackage{ifpdf}

\ifpdf
  \usepackage[colorlinks,final,backref=page,hyperindex]{hyperref}
\else
  \usepackage[colorlinks,final,backref=page,hyperindex,hypertex]{hyperref}
\fi
\usepackage{tikz}
\usepackage[active]{srcltx}

\makeatletter

\newtheorem{thm}{Theorem}[section]
\newtheorem{lem}[thm]{Lemma}
\newtheorem{cor}[thm]{Corollary}
\newtheorem{pro}[thm]{Proposition}
\newtheorem{ex}[thm]{Example}
\theoremstyle{definition}
\newtheorem{rmk}[thm]{Remark}
\newtheorem{defi}[thm]{Definition}

\newcommand{\nc}{\newcommand}
\newcommand{\delete}[1]{}

\nc{\mlabel}[1]{\label{#1}}  % Use this to suppress names
\nc{\mcite}[1]{\cite{#1}}  % Use this to suppress names
\nc{\mref}[1]{\ref{#1}}  % Use this to suppress names
\nc{\mbibitem}[1]{\bibitem{#1}} % Use this to show number
\delete{% Use the next two lines to show names
\nc{\mlabel}[1]{\label{#1}{\hfill \hspace{1cm}{\bf{{\ }\hfill(#1)}}}}
\nc{\mcite}[1]{\cite{#1}{{\bf{{\ }(#1)}}}}  % Use this lines to show names
\nc{\mref}[1]{\ref{#1}{{\bf{{\ }(#1)}}}}  % Use this lines to show names
\nc{\mbibitem}[1]{\bibitem[\bf #1]{#1}} % Use this to show name
}

\newcommand {\emptycomment}[1]{}

\delete{
\setlength{\baselineskip}{1.8\baselineskip}

\newcommand{\emptycomment}[1]{}

}

\nc{\calo}{\mathcal{O}}
\nc{\oop}{$\mathcal{O}$-operator\xspace}
\nc{\oops}{$\mathcal{O}$-operators\xspace}
\nc{\mrho}{{\bm{\varrho}}}
\nc{\bfk}{\mathbf{K}}
\nc{\invlim}{\displaystyle{\lim_{\longleftarrow}}\,}
\nc{\ot}{\otimes}
\nc{\CV}{\mathbf{C}}

\newcommand{\dr}{\dM^{\rm reg}}
\newcommand{\add}{\frka\frkd}

\newcommand{\lon }{\,\rightarrow\,}
\newcommand{\be }{\begin{equation}}
\newcommand{\ee }{\end{equation}}

\newcommand{\g}{\mathfrak g}
\newcommand{\h}{\mathfrak h}

\newcommand{\huaB}{\mathcal{B}}%{{\mathcal{E}}}%{\mathcal{B}}

\newcommand{\huaR}{\mathcal{R}}

\newcommand{\huaC}{{\mathcal{C}}}%{\mathcal{C}}

\newcommand{\huaH}{\mathcal{H}}

\newcommand{\huaO}{{\mathcal{O}}}

\newcommand{\huaZ}{\mathcal{Z}}

\newcommand{\frka}{\mathfrak a}

\newcommand{\frkd}{\mathfrak d}

\newcommand{\frkh}{\mathfrak h}

\newcommand{\frkT}{\mathfrak T}

\newcommand{\half}{\frac{1}{2}}

\newcommand{\Courant}[1]{\left\llbracket  #1\right\rrbracket }

\newcommand{\Id}{{\rm{Id}}}

\newcommand{\br}[1]{   [ \cdot,    \cdot  ]   }

\newcommand{\dM}{\mathrm{d}}

\newcommand{\Hom}{\mathrm{Hom}}

\newcommand{\Nij}{\mathrm{Nij}}
\newcommand{\Ob}{\mathsf{Ob^2_{T}}}

\newcommand{\gl}{\mathfrak {gl}}

\newcommand{\ad}{\mathrm{ad}}

\newcommand{\K}{\mathbf{K}}

\begin{document}

\title[Deformation and cohomology of $\huaO$-operators]{Deformations and their controlling cohomologies of $\huaO$-operators}

\author{Rong Tang}
\address{Department of Mathematics, Jilin University, Changchun 130012, Jilin, China}
\email{tangrong16@mails.jlu.edu.cn}

\author{Chengming Bai}
\address{Chern Institute of Mathematics and LPMC, Nankai University,
Tianjin 300071, China}
\email{baicm@nankai.edu.cn}

\author{Li Guo}
\address{Department of Mathematics and Computer Science,
         Rutgers University,
         Newark, NJ 07102}
\email{liguo@rutgers.edu}

\author{Yunhe Sheng}
\address{Department of Mathematics, Jilin University, Changchun 130012, Jilin, China}
\email{shengyh@jlu.edu.cn}

\date{\today}

\begin{abstract}
\oops are important in broad areas in mathematics and physics,
such as integrable systems, the classical Yang-Baxter equation,
pre-Lie algebras and splitting of operads. In this paper, a
deformation theory of \oops is established in
consistence with the general principles of deformation theories.
On the one hand, \oops are shown to be characterized as the
Maurer-Cartan elements in a suitable  graded Lie algebra. A given
\oop gives rise to a differential graded Lie algebra whose
Maurer-Cartan elements characterize deformations of the given
\oop. On the other hand, a Lie algebra with a representation is
identified from an \oop $T$ such that the corresponding
Chevalley-Eilenberg cohomology controls deformations of $T$, thus
can be regarded as an analogue of the Andr\'e-Quillen cohomology
for the \oop. Thereafter, infinitesimal and formal deformations of
\oops are studied. In particular, the notion of Nijenhuis elements
is introduced to characterize trivial infinitesimal deformations.
Formal deformations and extendibility of   order $n$ deformations
of an \oop are also characterized in terms of the new cohomology
theory. Applications are given to deformations of Rota-Baxter
operators of weight 0 and skew-symmetric $r$-matrices for
the classical Yang-Baxter equation. For skew-symmetric
$r$-matrices, there is an independent Maurer-Cartan
characterization of the deformations as well as an analogue of the
Andr\'e-Quillen cohomology controlling the deformations, which
turn out to be equivalent to the ones obtained as \oops associated
to the coadjoint representations. Finally, infinitesimal
deformations of skew-symmetric $r$-matrices and their
corresponding triangular Lie bialgebras are studied.
\end{abstract}

\subjclass[2010]{17B37,81R50,17B56,81R12,16T26,17A30,17B62}

\keywords{cohomology, deformation, $\huaO$-operator, Rota-Baxter operator, $r$-matrix, pre-Lie algebra}

\maketitle

\vspace{-1.1cm}

\tableofcontents

\allowdisplaybreaks

\section{Introduction}\mlabel{sec:intr}
This paper studies  deformations of \oops, in particular of Rota-Baxter operators and skew-symmetric $r$-matrices, using Maurer-Cartan elements and cohomology theory.

\subsection{Deformations and cohomology}
\label{ss:deform}
The method of deformation is ubiquitous in mathematics and physics.
Roughly speaking, a deformation of an object in a mathematical structure is a perturbation of the object (by a parameter for instance) which gives the same  kind of   structure. Motivated by the foundational  work of Kodaira and Spencer~\cite{KS} for complex analytic structures, deformation theory finds its generalization in algebraic geometry~\cite{Ha} and further in number theory as deformations of Galois representations~\cite{Maz}.

In physics, the idea of deformation is behind the perturbative quantum field theory and quantizing classical mechanics, inspiring the mathematical notion of quantum groups. Deformation quantization has been studied under many contexts in mathematical physics~\cite{K1,K2,Ri,Sc}.

The deformation of algebraic structures began with the seminal
work of Gerstenhaber~\cite{Ge0,Ge,Ge2,Ge3,Ge4} for associative
algebras and followed by its extension to Lie algebras by
Nijenhuis and Richardson~\cite{NR,NR2}. Deformations of other
algebraic structures such as pre-Lie algebras have also been
developed~\cite{Bu0}. In general, deformation theory was developed
for binary quadratic operads by Balavoine~\cite{Bal}. For more
general operads we refer the reader to the books of
Kontsevich-Soibelman~\cite{KSo} and Loday-Vallette~\cite{LV}, and the references therein. Also see
the paper of Fox~\cite{Fo} for a categorical approach by triples
and cotriples.

A suitable deformation theory of an algebraic structure needs to
follow certain general principle: on the one hand, for a given
object with the algebraic structure, there should be a
differential graded Lie algebra whose Maurer-Cartan elements
characterize deformations of this object. On the other hand, there
should be a suitable cohomology so that the infinitesimal of a
formal deformation can be identified with a cohomology class, and
then a theory of the obstruction to the integration of an
infinitesimal deformation can be developed using this cohomology
theory. The cohomology groups for the deformation theories
of associative algebras and Lie algebras are the
Hochschild cohomology groups and the Chevalley-Eilenberg
cohomology groups respectively. In general the cohomology
groups are the Andr\'e-Quillen cohomology groups which are
isomorphic to the cohomology groups of the deformation
complexes~\cite{LV}.

Nijenhuis operators also play an important role in  deformation
theories  due to their relationship with trivial infinitesimal
deformations. There are interesting applications of Nijenhuis operators such as constructing biHamiltonian systems to study the integrability of nonlinear evolution equations \cite{CGM,Do}.

\subsection{Rota-Baxter operators, skew-symmetric $r$-matrices and \oops}
The above deformation theories do not apply to the study of  deformations of linear operators on algebras such as Rota-Baxter operators and more generally \oops, as well as skew-symmetric $r$-matrices.
The goal of this paper is to develop a deformation theory of \oops.

We first recall some basic concepts.

\begin{defi} \label{defi:O} Let $(\g,[\cdot,\cdot])$ be a Lie algebra.
\begin{enumerate}
\item[\rm(i)]
Let $\lambda$ be a scalar. A linear operator $P:\g\longrightarrow \g$ is called a {\bf Rota-Baxter operator of weight $\lambda$} if
\begin{equation} [P(x),P(y)]=P\big([P(x),y]+ [x,P(y)] +\lambda  [x,y]\big), \quad \forall x, y \in \g.
\label{eq:rbo}
\end{equation}
\item[\rm(ii)] We also use the notation $[\cdot,\cdot]$ to denote
the Gerstenhaber bracket on $\wedge^\bullet\g$. An element
$r\in\wedge^2\g$ is called a skew-symmetric {\bf $r$-matrix} if $r$ satisfies the {\bf classical Yang-Baxter equation (CYBE)}:
\begin{equation}\label{eq:cybe}
  [r,r]=0.
\end{equation}
\item[\rm(iii)]
Let $\rho:\g\longrightarrow\gl(V)$ be a representation of $\g$ on a vector space $V$. An {\bf \oop} on $\g$ with respect to the representation $(V;\rho)$ is a linear map $T:V\longrightarrow\g$ such that
 \begin{equation}
   [Tu,Tv]=T\big(\rho(Tu)(v)-\rho(Tv)(u)\big),\quad\forall u,v\in V.
 \mlabel{eq:defiO}
 \end{equation}
\end{enumerate}
\label{de:conc}
\end{defi}
Note that when $\rho$ is the adjoint representation of $\g$,
Eq.~\eqref{eq:defiO} reduces to Eq.~(\ref{eq:rbo}) with
$\lambda=0$, which means that a Rata-Baxter operator of weight
zero is an \oop on $\g$ with respect to the adjoint
representation. Furthermore, a skew-symmetric $r$-matrix
corresponds to an $\mathcal O$-operator on $\g$ with respect to
the coadjoint representation \cite{Ku}.

The concept of Rota-Baxter operators on associative algebras was introduced in 1960 by G. Baxter \cite{Ba} in his study of
fluctuation theory in probability. Recently it has found many applications, including in Connes-Kreimer's~\cite{CK} algebraic approach to the renormalization in perturbative quantum
field theory. In the Lie algebra context, a Rota-Baxter
operator of weight zero
was introduced independently in the
1980s as the operator form of the classical Yang-Baxter equation,
named after the physicists C.-N. Yang and R. Baxter \cite{BaR,Ya}, whereas the classical Yang-Baxter equation plays important roles in many fields in mathematics and mathematical physics such as integrable systems and quantum groups \cite{CP,STS}. For further details on Rota-Baxter operators, see~\cite{Gub}.

To better understand the classical Yang-Baxter equation and
the related integrable systems, the more general notion of an \oop (later also called
a relative Rota-Baxter operator or a generalized Rota-Baxter operator)
on a Lie algebra was introduced by Kupershmidt~\cite{Ku},
which can be traced back to Bordemann
\cite{Bor}. In addition, the defining relation of an \oop was also
called the Schouten curvature and is the algebraic formulation of
the contravariant analogue of the Cartan curvature of a Lie
algebra-valued one-form on a Lie group~\cite{KM}. An
$\huaO$-operator gives rise to a skew-symmetric $r$-matrix in a larger Lie algebra \cite{Bai}.

In the context of associative algebras, \oops give rise to the
important structure of dendriform algebras (\cite{Lo5}) and, more
generally, leads lead to the splitting of
operads~\cite{BBGN,PBG}.

\subsection{Summary of the results and outline of the paper}

Given the critical roles played by Rota-Baxter operators, skew-symmetric $r$-matrices and \oops,
it is important to develop their deformation and cohomology theories. As
aforementioned, the existing general theories do not apply to such cases.
To meet this need, we establish a deformation theory of \oops which is remarkably consistent with
the general principles of deformation theories as indicated in Section~\ref{ss:deform}, including a suitable differential graded Lie algebra whose Maurer-Cartan elements characterize the \oops and their deformations as well as an analogue of the Andr\'e-Quillen cohomology which
controls the infinitesimal and formal deformations of \oops.
Furthermore, \oops are closely related to pre-Lie algebras~\cite{Bai} (see Definition~\ref{de:prelie}) which have a well-established deformation theory. Our deformation theory of \oops is also compatible with that of pre-Lie algebras.
We hope that this study will shed light on a general deformation theory for algebraic structures (operads) with nontrivial linear operators.

In the following we give a summary of the main results and an
outline of the paper.

First Section~\ref{ss:mce} provides the Maurer-Cartan
characterization of \oops and their deformations. From a
representation $(V;\rho)$ of a Lie algebra $\g$, we obtain, via
the derived bracket, a graded Lie algebra
$(\oplus_{k=0}^{\dim(V)}\Hom(\wedge^kV,\g),\Courant{\cdot,\cdot})$,
of which the Maurer-Cartan elements are exactly the
$\huaO$-operators. Further, a given \oop $T$, as a Maurer-Cartan
element, gives rise to a differential $d_T:=\Courant{T,\cdot}$ on
this graded Lie algebra. The Maurer-Cartan elements of the
resulting differential graded Lie algebra correspond precisely to
deformations of the given \oop $T$. There is a close
relationship between this graded Lie algebra and the one for
pre-Lie algebras given in \cite{CL}.

Section~\ref{sec:coh} sets up a cohomology theory for \oops. In
order to obtain a suitable analogue of the Andr\'e-Quillen
cohomology for \oops, it is natural to take the
Chevalley-Eilenberg cohomology of a Lie algebra with coefficients
in a representation. Contrary to our intuition, it is not the Lie
algebra $\g$ and the representation $V$, but a new Lie algebra
structure on $V$ induced by the \oop $T$ and a representation of
$V$ on $\g$. Explicitly, associated to an \oop $T$ on a Lie
algebra $\g$ with respect to a representation $(V;\rho)$ we obtain a Lie algebra
$V^c:=(V, [\cdot,\cdot]_T)$ with $[u,v]_T:=\rho(Tu)(v)-\rho(Tv)(u)
\text{ for all } u, v\in V,$ and identify a natural representation
$\mrho$ of the Lie algebra $V^c$ on the space $\g$. We take the
corresponding Chevalley-Eilenberg cohomology to be the cohomology
of the \oop and apply it to control infinitesimal and formal
deformations   of $\huaO$-operators in the following sections.
Moreover, we found that this Lie algebra $V^c$ is exactly the
commutator of the pre-Lie algebra induced by the \oop $T$. There
is also a natural map $\Phi$ from these cohomology groups to the
cohomology groups of the associated pre-Lie algebra.

The usual isomorphism~\cite{LV} between the Andr\'e-Quillen cohomology and the cohomology of the deformation complex has its counterpart for \oops: this Chevalley-Eilenberg coboundary operator $d_\mrho$ coincides with the differential $d_T$ introduced above up to a sign, completing the following diagram:
\vspace{-.3cm}
 \[
\small{ \xymatrix{
\huaO\mbox{-operator} ~T \ar[d] \ar[r] & \mbox{MC element}~ T \ar[l] \ar[r] & \mbox{differential} ~d_T=\Courant{T,\cdot} \ar@{=}[d]^{(-1)^{|\cdot|}} \\
 \mbox{Lie algebra } V^c \ar[r] & \mbox{rep}~ (\g;\mrho)~  \mbox{of }~ V^c \ar[r] & \mbox{C.E. coboundary operator} ~ d_\mrho.}
}
\vspace{-.3cm}
\]

Section~\ref{sec:infdef} studies one parameter infinitesimal
deformations of an \oop. We show that if two deformations are
isomorphic, then the corresponding generators are in the same
cohomology class of the \oop. We introduce the notion of Nijenhuis
elements to characterize trivial deformations. By means of the
above natural map $\Phi$ from cohomology groups of the \oop and
those of the corresponding associated pre-Lie algebra,
further relations are obtained:
\smallskip

\begin{tabular}{|c|c|}
\hline  $\huaO$-operators $T:V\rightarrow\g$ w.r.t. $(V;\rho)$& pre-Lie algebra $(V,\cdot_T),~u\cdot_Tv:=\rho(Tu)(v)$ \\\hline \hline
infinitesimal deformation of $T$ & infinitesimal deformation of $(V,\cdot_T)$ \\
generated by $\frkT:V\rightarrow\g$, $T+t\frkT$  & generated by   $\omega_\frkT$, $\omega_\frkT(u,v):=\rho(\frkT u)(v)$\\\hline
$\frkT$ is a 1-cocycle& $\omega_\frkT$ is a 2-cocycle \\\hline
trivial deformations correspond to & trivial deformations correspond to\\
Nijenhuis elements $x\in\g$& Nijenhuis operators $\rho(x)\in\gl(V)$ on $(V,\cdot_T)$
\\\hline
\end{tabular}
\smallskip

Section~\ref{sec:fordef} utilizes the cohomology theory of an \oop to study formal deformations of \oops.
We show that the infinitesimals of two
equivalent one-parameter formal deformations of an \oop are in the
same first cohomology class of the \oop and that a higher order
deformation of an \oop is extendable if and only if its
obstruction class in the second cohomology group of the \oop is
trivial.

Section~\ref{sec:rbar} specializes to Rota-Baxter operators of
weight 0 on a Lie algebra $\g$, regarded as $\huaO$-operators on
$\g$ with respect to the adjoint representation. We give some
precise formulas for deformations of Rota-Baxter operators of
weight 0. Nijenhuis elements in certain Rota-Baxter Lie
algebras (Lie algebras with Rota-Baxter operators of weight 0) are classified.

Section~\ref{sec:rmat} focuses on deformations of skew-symmetric
$r$-matrices on a Lie algebra $\g$, regarded as 
$\huaO$-operators on $\g$ with respect to the coadjoint
representation. Viewing the CYBE as a Maurer-Cartan equation, we
first provide a direct Maurer-Cartan characterization of
deformations and an analogue of the Andr\'e-Quillen cohomology
controlling the infinitesimal deformations. This deformation
theory turns out to be equivalent to the one obtained as \oops
with respect to the coadjoint representation. Through this
equivalence, a notion of weak homomorphism 
between skew-symmetric $r$-matrices is introduced to further study
their infinitesimal deformations. Finally, we study infinitesimal
deformations of triangular Lie bialgebras by the natural
correspondence between skew-symmetric $r$-matrices and triangular
Lie bialgebras.

\smallskip
Throughout this paper, we work over an algebraically closed field $\K$ of characteristic 0 and all the vector spaces are over $\K$ and are finite-dimensional. 
\section{Maurer-Cartan elements, \oops and their deformations}
\label{ss:mce}

Usually for an algebraic structure, Maurer-Cartan elements in a
suitable graded Lie algebra are used to characterize realizations
of the algebraic structure on a space. For a given realization of
the algebraic structure, the corresponding Maurer-Cartan element
equipped the graded Lie algebra with a differential. Then the
deformations of the given realization are characterized as the
Maurer-Cartan elements of the resulting differential graded Lie
algebra. See Remark~\ref{rk:plmc} for the case of pre-Lie algebras
and~\cite{LV} for operads. Adapting this principle to \oops, we
first need to construct a  graded Lie algebra for a Lie algebra with a
representation whose Maurer-Cartan elements characterize the
\oops. It then follows that a given \oop gives rise to a differential on this graded Lie algebra and there is a one-to-one
correspondence between the set of Maurer-Cartan elements in the resulting
differential graded Lie algebra and the set of deformations of this \oop.

We first recall a general notion and a basic fact~\cite{LV}.
\begin{defi}
  Let $(\g=\oplus_{k=0}^\infty\g_i,[\cdot,\cdot],\dM)$ be a differential graded Lie algebra.  A degree $1$ element $\theta\in\g_1$ is called a {\bf Maurer-Cartan element} of $\g$ if it
  satisfies the following {\bf Maurer-Cartan equation}:
\vspace{-.3cm}
  \begin{equation}
  \dM \theta+\half[\theta,\theta]=0.
  \label{eq:mce}
  \end{equation}
  \end{defi}
A graded Lie algebra is a differential graded Lie
algebra with $d=0$. Then we
have

\begin{pro}
Let $(\g=\oplus_{k=0}^\infty\g_i,[\cdot,\cdot])$ be a graded Lie algebra and let $\mu\in \g_1$ be a Maurer-Cartan element. Then the map
$$ d_\mu: \g \longrightarrow \g, \ d_\mu(u):=[\mu, u], \quad \forall u\in \g,$$
is a differential on $\g$. For any $v\in \g$, the sum $\mu+v$ is a
Maurer-Cartan element of the graded Lie algebra $(\g,
[\cdot,\cdot])$ if and only if $v$ is a Maurer-Cartan element of the differential graded Lie algebra $(\g, [\cdot,\cdot], d_\mu)$. \label{pp:mce}
\end{pro}

Let $(V;\rho)$ be a representation of a Lie algebra $\g$. Consider the graded vector space
$$\huaC^*(V,\g):=\oplus_{k=0}^{\dim(V)}\Hom(\wedge^{k}V,\g).$$
Define a skew-symmetric bracket operation $$\Courant{\cdot,\cdot}: \Hom(\wedge^nV,\g)\times \Hom(\wedge^mV,\g)\longrightarrow \Hom(\wedge^{m+n}V,\g)$$ by
\begin{eqnarray}
&&\nonumber\Courant{P,Q}(u_1,u_2,\cdots,u_{m+n})\\
\mlabel{o-bracket}&=&\sum_{\sigma\in \mathbb S_{(m,1,n-1)}}(-1)^{\sigma}P(\rho(Q(u_{\sigma(1)},\cdots,u_{\sigma(m)}))u_{\sigma(m+1)},u_{\sigma(m+2)},\cdots,u_{\sigma(m+n)})\\
\nonumber&&-(-1)^{mn}\sum_{\sigma\in \mathbb S_{(n,1,m-1)}}(-1)^{\sigma}Q(\rho(P(u_{\sigma(1)},\cdots,u_{\sigma(n)}))u_{\sigma(n+1)},u_{\sigma(n+2)},\cdots,u_{\sigma(m+n)})\\
\nonumber&&+(-1)^{mn}\sum_{\sigma\in \mathbb S_{(n,m)}}(-1)^{\sigma}[P(u_{\sigma(1)},\cdots,u_{\sigma(n)}),Q(u_{\sigma(n+1)},\cdots,u_{\sigma(m+n)})]
\end{eqnarray}
for all $P\in\Hom(\wedge^nV,\g)$ and $Q\in\Hom(\wedge^mV,\g)$.

Note that for all $x,y\in\g$, $\Courant{x,y}=[x,y]$. Furthermore, we have
\begin{pro}\mlabel{pro:gla}
  With the above notations, $(\huaC^*(V,\g),\Courant{\cdot,\cdot})$ is a graded Lie algebra. Its Maurer-Cartan elements are precisely the $\huaO$-operators on $\g$ with respect to the representation $(V;\rho)$.
\end{pro}
\begin{proof}
In short, the graded Lie algebra $(\huaC^*(V,\g),\Courant{\cdot,\cdot})$ is obtained via the derived bracket \cite{Vo}. In fact, the Nijenhuis-Richardson bracket $[\cdot,\cdot]_{NR}$ associated to the direct sum vector space $\g\oplus V$ gives rise to a graded Lie algebra $(\oplus_{k=0}^{\dim(\g\oplus V)}\Hom(\wedge^k(\g\oplus V),\g\oplus V),[\cdot,\cdot]_{NR})$. Obviously $\oplus_{k=0}^{\dim(V)}\Hom(\wedge^kV,\g)$ is an abelian subalgebra. A linear map $\mu:\wedge^2\g\longrightarrow \g$ is a Lie algebra structure and $\rho:\g\otimes V\longrightarrow V$ is a representation of $\g$ on $V$ if and only if $\mu+\rho$ is a Maurer-Cartan element of the graded Lie algebra $(\oplus_{k=0}^{\dim(\g\oplus V)}\Hom(\wedge^k(\g\oplus V),\g\oplus V),[\cdot,\cdot]_{NR})$, defining a differential $d_{\mu+\rho}$ on $(\oplus_{k=0}^{\dim(\g\oplus V)}\Hom(\wedge^k(\g\oplus V),\g\oplus V),[\cdot,\cdot]_{NR})$ via
  $$
  d_{\mu+\rho}=[\mu+\rho,\cdot]_{NR}.
  $$
Further, the differential $d_{\mu+\rho}$ gives rise to a graded Lie algebra structure on the graded vector space $\oplus_{k=0}^{\dim(V)}\Hom(\wedge^kV,\g)$ via the derived bracket
  $$
  \Courant{P,Q}:=(-1)^{n}[[\mu+\rho,P]_{NR},Q]_{NR},\quad\forall P\in\Hom(\wedge^nV,\g), Q\in\Hom(\wedge^mV,\g),
  $$
which is exactly the bracket given by Eq.~\eqref{o-bracket}.

Finally, for a degree one element $T:V\longrightarrow \g$, Eq.~\eqref{o-bracket} becomes
$$\Courant{T,T}(u_1,u_2) = 2\big(T(\rho(Tu_1)u_2)-T(\rho(Tu_2)u_1)-[Tu_1,Tu_2]\big),\quad \forall u_1,u_2\in V.$$
This proves the last statement.
\end{proof}

Let $T:V\longrightarrow\g$ be an \oop. Since $T$ is a Maurer-Cartan element of the graded Lie algebra $(\huaC^*(V,\g),\Courant{\cdot,\cdot})$ by Proposition~\ref{pro:gla}, it follows from Proposition~\ref{pp:mce} that
$d_T:=\Courant{T,\cdot}$
 is a graded derivation on the graded Lie
algebra $(\huaC^*(V,\g),\Courant{\cdot,\cdot})$ satisfying $d^2_T=0$.
  Therefore $(\huaC^*(V,\g),\Courant{\cdot,\cdot},d_T)$ is a differential graded Lie algebra.
Further

\begin{thm}\label{thm:deformation}
Let $T:V\longrightarrow\g$ be an $\huaO$-operator on a Lie algebra
$\g$ with respect to a representation $(V;\rho)$. Then for a
linear map $T':V\longrightarrow \g$, $T+T'$ is still
an \oop on the Lie algebra $\g$ associated to the
representation $(V;\rho)$ if and only if $T'$ is a Maurer-Cartan
element of the differential graded Lie algebra
$(\huaC^*(V,\g),\Courant{\cdot,\cdot},d_T)$.
\end{thm}

We next recall the notion of a pre-Lie algebra and the differential graded Lie algebra whose Maurer-Cartan elements  characterize pre-Lie algebra structures. We show that there is a close relationship between these two differential graded Lie algebras.

\begin{defi}\label{de:prelie}
  A {\bf pre-Lie algebra} is a pair $(V,\cdot_V)$, where $V$ is a vector space and  $\cdot_V:V\otimes V\longrightarrow V$ is a bilinear multiplication
satisfying that for all $x,y,z\in V$, the associator
$$(x,y,z):=(x\cdot_V y)\cdot_V z-x\cdot_V(y\cdot_V z)$$
is symmetric in $x,y$, that is,
$$(x,y,z)=(y,x,z)\;\;{\rm or}\;\;{\rm
equivalently,}\;\;(x\cdot_V y)\cdot_V z-x\cdot_V(y\cdot_V z)=(y\cdot_V x)\cdot_V
z-y\cdot_V(x\cdot_V z).$$

\end{defi}

Relating an \oop to a pre-Lie algebra, we have
\begin{thm} $($\cite{Bai}$)$
Let $T:V\to \g$ be an $\huaO$-operator on a Lie algebra $\g$ with respect to a representation $(V;\rho)$. Define a multiplication $\cdot_T$ on $V$ by
\begin{equation}
  u\cdot_T v=\rho(Tu)(v),\quad \forall u,v\in V.
\end{equation}
Then $(V,\cdot_T)$ is a pre-Lie algebra.
\mlabel{thm:opL}
\end{thm}

A permutation $\sigma\in\mathbb S_n$ is called an $(i,n-i)$-unshuffle if $\sigma(1)<\cdots <\sigma(i)$ and $\sigma(i+1)<\cdots <\sigma(n)$. If $i=0$ or $n$, we assume $\sigma=\Id$. The set of all $(i,n-i)$-unshuffles will be denoted by $\mathbb S_{(i,n-i)}$. The notion of an $(i_1,\cdots,i_k)$-unshuffle and the set $\mathbb S_{(i_1,\cdots,i_k)}$ are defined analogously.

Let $V$ be a vector space. For $\alpha\in\Hom(\wedge^{n}V\otimes V,V)$ and $\beta\in\Hom(\wedge^{m}V\otimes V,V)$, define  $\alpha\circ\beta\in\Hom(\wedge^{n+m}V\otimes V,V)$ by
\begin{eqnarray}
\nonumber&&(\alpha\circ\beta)(u_1,\cdots,u_{m+n+1})\\
\mlabel{eq:pLbrac}&:=&\sum_{\sigma\in\mathbb S_{(m,1,n-1)}}(-1)^{\sigma}\alpha(\beta(u_{\sigma(1)},\cdots,u_{\sigma(m+1)}),u_{\sigma(m+2)},\cdots,u_{\sigma(m+n)},u_{m+n+1})\\
&&+(-1)^{mn}\sum_{\sigma\in\mathbb S_{(n,m)}}(-1)^{\sigma}\alpha(u_{\sigma(1)},\cdots,u_{\sigma(n)},\beta(u_{\sigma(n+1)},\cdots,u_{\sigma(m+n)},u_{m+n+1})).
\nonumber
\end{eqnarray}
Then the graded vector space
$\CV^*(V,V):=\oplus_{k\geq 0} \Hom(\wedge^kV\ot V,V)$ equipped with the graded bracket
\begin{eqnarray}
[\alpha,\beta]^C:=\alpha\circ\beta-(-1)^{mn}\beta\circ\alpha,
\quad \forall \alpha\in\Hom(\wedge^{n}V\otimes
V,V),\beta\in\Hom(\wedge^{m}V\otimes V,V),
\mlabel{eq:glapL}
\end{eqnarray}
is a graded Lie algebra.
See \cite{CL,WBLS} for more details.

\begin{rmk}
For $\alpha\in\Hom(V\otimes V,V)$, we have
\begin{eqnarray*}
[\alpha,\alpha]^C(u,v,w)=2(\alpha\circ\alpha)(u,v,w)=2\big(\alpha(\alpha(u,v),w)-\alpha(\alpha(v,u),w)-\alpha(u,\alpha(v,w))+\alpha(v,\alpha(u,w))\big).
\end{eqnarray*}
Thus, $\alpha$ defines a pre-Lie algebra structure on $V$ if and only if $[\alpha,\alpha]^C=0,$ that is, $\alpha$ is a Maurer-Cartan element  of the graded Lie algebra $(\CV^*(V,V),[\cdot,\cdot]^C)$.
\label{rk:plmc}
\end{rmk}

Define a linear map $\Phi:\Hom(\wedge^kV,\g)\lon\Hom(\wedge^{k}V\otimes V,V), k\geq 0,$   by
\begin{eqnarray}\label{eq:phi}
\Phi(f)(u_1,\cdots,u_k,u_{k+1})=\rho(f(u_1,\cdots,u_k))(u_{k+1}),\,\,\,\,\forall
f\in\Hom(\wedge^kV,\g), u_1,\cdots,u_{k+1}\in V.
\mlabel{eq:defiphi}
\end{eqnarray}

\begin{pro}
Let $(V;\rho)$ be a representation of a Lie algebra $\g$. Then $\Phi$ is a
homomorphism of graded Lie algebras from
$(\huaC^*(V,\g),\Courant{\cdot,\cdot})$ to
$(\CV^*(V,V),[\cdot,\cdot]^C)$.
\end{pro}

\begin{proof}
For $P\in\Hom(\wedge^nV,\g)$ and $Q\in\Hom(\wedge^mV,\g)$, we have $[\Phi(P),\Phi(Q)]^C\in\Hom(\wedge^{m+n}V\otimes V,V)$. More precisely, for all $u_1,\cdots,u_{m+n+1}\in V$, we have
\begin{eqnarray*}
&&(\Phi(P)\circ\Phi(Q))(u_1,\cdots,u_{m+n+1})\\
&=&\sum_{\sigma\in\mathbb S_{(m,1,n-1)}}(-1)^{\sigma}\Phi(P)(\Phi(Q)(u_{\sigma(1)},\cdots,u_{\sigma(m+1)}),u_{\sigma(m+2)},\cdots,u_{\sigma(m+n)},u_{m+n+1})\\
&&+(-1)^{mn}\sum_{\sigma\in\mathbb S_{(n,m)}}(-1)^{\sigma}\Phi(P)(u_{\sigma(1)},\cdots,u_{\sigma(n)},\Phi(Q)(u_{\sigma(n+1)},\cdots,u_{\sigma(m+n)},u_{m+n+1}))\\
&=&\sum_{\sigma\in\mathbb S_{(m,1,n-1)}}(-1)^{\sigma}\rho\big(P(\rho(Q(u_{\sigma(1)},\cdots,u_{\sigma(m)}))u_{\sigma(m+1)},u_{\sigma(m+2)},\cdots,u_{\sigma(m+n)})\big)u_{m+n+1}\\
&&+(-1)^{mn}\sum_{\sigma\in\mathbb S_{(n,m)}}(-1)^{\sigma}\rho(P(u_{\sigma(1)},\cdots,u_{\sigma(n)}))\rho(Q(u_{\sigma(n+1)},\cdots,u_{\sigma(m+n)}))u_{m+n+1}.
\end{eqnarray*}
For any $\sigma\in\mathbb S_{(n,m)}$, we define $\tau\in\mathbb S_{(m,n)}$ by
\[
\tau(i)=\left\{
\begin{array}{ll}
\sigma(n+i) & 1\le i\le m;\\
\sigma(i-m) & m+1\le i\le m+n.
\end{array}
\right.
\]
Thus we have $(-1)^{\tau}=(-1)^{mn}(-1)^{\sigma}$.
In fact, the elements of $\mathbb S_{(n,m)}$ are in bijection with the elements of $\mathbb S_{(m,n)}$. Then we have
\begin{eqnarray*}
&&(\Phi(Q)\circ\Phi(P))(u_1,\cdots,u_{m+n+1})\\
&=&\sum_{\tau\in\mathbb S_{(n,1,m-1)}}(-1)^{\tau}\rho\big(Q(\rho(P(u_{\tau(1)},\cdots,u_{\tau(n)}))u_{\tau(n+1)},u_{\tau(n+2)},\cdots,u_{\tau(m+n)})\big)u_{m+n+1}\\
&&+(-1)^{mn}\sum_{\tau\in\mathbb S_{(m,n)}}(-1)^{\tau}\rho(Q(u_{\tau(1)},\cdots,u_{\tau(m)}))\rho(P(u_{\tau(m+1)},\cdots,u_{\tau(m+n)}))u_{m+n+1}\\
&=&\sum_{\sigma\in\mathbb S_{(n,1,m-1)}}(-1)^{\sigma}\rho\big(Q(\rho(P(u_{\sigma(1)},\cdots,u_{\sigma(n)}))u_{\sigma(n+1)},u_{\sigma(n+2)},\cdots,u_{\sigma(m+n)})\big)u_{m+n+1}\\
&&+\sum_{\sigma\in\mathbb S_{(n,m)}}(-1)^{\sigma}\rho(Q(u_{\sigma(n+1)},\cdots,u_{\sigma(n+m)}))\rho(P(u_{\sigma(1)},\cdots,u_{\sigma(n)}))u_{m+n+1}.
\end{eqnarray*}
Therefore, we have
\begin{eqnarray*}
&&[\Phi(P),\Phi(Q)]^C(u_1,\cdots,u_{m+n+1})\\
&=&\sum_{\sigma\in\mathbb S_{(m,1,n-1)}}(-1)^{\sigma}\rho\big(P(\rho(Q(u_{\sigma(1)},\cdots,u_{\sigma(m)}))u_{\sigma(m+1)},u_{\sigma(m+2)},\cdots,u_{\sigma(m+n)})\big)u_{m+n+1}\\
&&-(-1)^{mn}\sum_{\sigma\in\mathbb S_{(n,1,m-1)}}(-1)^{\sigma}\rho\big(Q(\rho(P(u_{\sigma(1)},\cdots,u_{\sigma(n)}))u_{\sigma(n+1)},u_{\sigma(n+2)},\cdots,u_{\sigma(m+n)})\big)u_{m+n+1}\\
&&+(-1)^{mn}\sum_{\sigma\in\mathbb S_{(n,m)}}(-1)^{\sigma}\rho([P(u_{\sigma(1)},\cdots,u_{\sigma(n)}),Q(u_{\sigma(n+1)},\cdots,u_{\sigma(m+n)})])u_{m+n+1}\\
&=&\Phi(\Courant{P,Q})(u_1,\cdots,u_{m+n+1}).
\end{eqnarray*}
Thus $\Phi$ is a homomorphism of graded Lie
algebras from $(\huaC^*(V,\g),\Courant{\cdot,\cdot})$ to
$(\CV^*(V,V),[\cdot,\cdot]^C)$.
\end{proof}

\begin{rmk}\label{rmk:O-pre}
As a direct consequence of the above proposition, the
Maurer-Cartan elements in the first graded Lie algebra are sent to
those in the second graded Lie algebra. Thus by
Proposition~\ref{pro:gla} and Remark~\ref{rk:plmc}, the \oops on $V$
are sent to pre-Lie algebra structures on $V$. Further two \oops on $V$ are sent to the same pre-Lie algebra if and only if they are in the same fiber of $\Phi$. This gives a strengthened form of Theorem~\ref{thm:opL}.
\end{rmk}

\section{Cohomology  of $\huaO$-operators}
\mlabel{sec:coh} In this section we   give a cohomology
theory for \oops which will be used to control
infinitesimal and formal deformations of   \oops in the following sections. Thus, this cohomology can be viewed as an analogue of the Andr\'e-Quillen cohomology.

Let $(V,\cdot_V)$ be a pre-Lie algebra. The commutator $
[x,y]^c=x\cdot_V y-y\cdot_V x$ defines a Lie algebra structure on
$V$, which is called the {\bf sub-adjacent Lie algebra} of
$(V,\cdot_V)$ and denoted by $V^c$. See \cite{Bu} for more details.
In particular, we denote by
$V^c:=(V,[\cdot,\cdot]_T)$ the sub-adjacent Lie algebra of the
pre-Lie algebra $(V,\cdot_T)$ induced from an \oop $T$ on a Lie algebra $\g$ with respect to a representation $(V;\rho)$ given in Theorem~\ref{thm:opL}. Then $T$ is a Lie algebra homomorphism from $(V,[\cdot,\cdot]_T)$ to $\g$.

Let $T:V\longrightarrow \g$ be an $\mathcal O$-operator on a Lie
algebra $\g$ with respect to a representation $(V;\rho)$. We
construct a representation of the sub-adjacent Lie algebra
$(V,[\cdot,\cdot]_T)$ on the vector space $\g$. We will show that the corresponding
Chevalley-Eilenberg cohomology will serve as the cohomology for \oops
that we have been looking for.

\begin{lem}\mlabel{lem:rep} 
Let $T$ be an $\mathcal O$-operator on a Lie
algebra $\g$ with respect to a representation $(V;\rho)$.
  Define
\vspace{-.2cm}
 \begin{equation}
\mrho=\mrho_T:V\longrightarrow\gl(\g), \quad     \mrho(u)(x):=[Tu,x]+T\rho(x)(u),\;\;\forall x\in \g,u\in
    V.
  \end{equation}
Then  $\mrho$ is a representation of the sub-adjacent Lie algebra $(V,[\cdot,\cdot]_T)$ on the vector space $\g$.
  \end{lem}

  \begin{proof}
  Since $(V;\rho)$ is a
representation of the Lie algebra $\g$, by
Eq.~\eqref{eq:defiO}, we have
  \vspace{-.1cm}
  \begin{eqnarray*}
    &&([\mrho(u),\mrho(v)]-\mrho([u,v]_T))(x)\\
    &=&\mrho(u)([Tv,x]+T\rho(x)(v))-\mrho(v)([Tu,x]+T\rho(x)u)-[T[u,v]_T,x]-T\rho(x)([u,v]_T)\\
    &=&[Tu,[Tv,x]]+T\rho([Tv,x])(u)+[Tu,T\rho(x)(v)]+T\rho(T\rho(x)(v))(u)\\
    &&-[Tv,[Tu,x]]-T\rho([Tu,x])(v)-[Tv,T\rho(x)(u)]-T\rho(T\rho(x)(u))(v)\\
    &&-[[Tu,Tv],x]-T\rho(x)\rho(Tu)(v)+T\rho(x)\rho(Tv)(u)\\
    &=&T\rho([Tv,x])(u)+[Tu,T\rho(x)(v)]+T\rho(T\rho(x)(v))(u) -T\rho([Tu,x])(v)\\
    &&-[Tv,T\rho(x)(u)]-T\rho(T\rho(x)(u))(v) -T\rho(x)\rho(Tu)(v)+T\rho(x)\rho(Tv)(u)\\
    &=&T[\rho(Tu),\rho(x)](v)-T\rho([Tu,x])(v)+T\rho([Tv,x])(u)-T[\rho(Tv),\rho(x)](u)\\
    &=&0. 
    \vspace{-.1cm}
  \end{eqnarray*}
  \vspace{-.1cm}
  Therefore, $\mrho$ is a representation.
  \end{proof}

\begin{rmk}
    Here we provide an intrinsic interpretation
    of the above representation $\mrho$ by means of the deformed Lie bracket by a Nijenhuis operator. First recall that a Nijenhuis operator on a Lie algebra $(\frkh,[\cdot,\cdot])$ is a linear map $N:\h\longrightarrow\h$ satisfying
    $$
    [Na,Nb]=N([Na,b]+[a,Nb]-N[a,b]),\quad\forall a,b\in\h.
    $$
    Then $(\h,[\cdot,\cdot]_N)$ is a Lie algebra, where the  bracket $[\cdot,\cdot]_N$ is given by
    $$
    [a,b]_N=[Na,b]+[a,Nb]-N[a,b],\;\;\forall a,b\in \h.
    $$
    Now it is straightforward to see that if $T:V\longrightarrow\g$ is an \oop on a Lie algebra $\g$ with respect to a representation $(V;\rho)$, then $N_T=\left(\begin{array}{cc}0&T\\0&0\end{array}\right)$ is a Nijenhuis operator
     on the semidirect product Lie algebra $\g\ltimes_\rho V$. Therefore there is a Lie algebra structure on $V\oplus \g\cong \g\oplus V$ defined by
    $$
    [x+u,y+v]_{N_T}=[u,v]_T+\mrho(u)(y)-\mrho(v)(x),
     \;\;\forall x,y\in \g, u,v\in V,
    $$
    which implies that $\mrho$ is a representation of the sub-adjacent Lie algebra $(V,[\cdot,\cdot]_T)$ on the vector space $\g$.
\vspace{-.1cm}  
\end{rmk}

Let $d_\mrho: \Hom(\wedge^kV,\g)\longrightarrow
\Hom(\wedge^{k+1}V,\g)$ be the corresponding Chevalley-Eilenberg
coboundary operator. More precisely, for all $f\in
\Hom(\wedge^kV,\g)$ and $u_1,\cdots,u_{k+1}\in V$, we have
\begin{eqnarray}
&& d_\mrho f(u_1,\cdots,u_{k+1})\notag\\
&:=&\sum_{i=1}^{k+1}(-1)^{i+1}[Tu_i,f(u_1,\cdots,\hat{u_i},\cdots, u_{k+1})]+\sum_{i=1}^{k+1}(-1)^{i+1}T\rho(f(u_1,\cdots,\hat{u_i},\cdots, u_{k+1}))(u_i)\label{eq:odiff}\\
&&+\sum_{1\le i<j\le k+1}(-1)^{i+j}f(\rho(Tu_i)(u_j)-\rho(Tu_j)(u_i),u_1,\cdots,\hat{u_i},\cdots,\hat{u_j},\cdots, u_{k+1}). \notag
\end{eqnarray}
It is obvious that $x\in\g$ is closed if and only if
         $$
           T\circ\rho(x)=\ad_x\circ T,
         $$
               and   $f\in \huaC^1(V,\g)$  is closed  if and only if
               $$
              [Tu,f(v)]-[Tv,f(u)]-T(\rho(f(u))(v)-\rho(f(v))(u))-f(\rho(Tu)(v)-\rho(Tv)(u))=0.
               $$

Comparing the coboundary operators $d_\mrho$ given above and the
operators $d_T=\Courant{T,\cdot}$ introduced in Theorem~\ref{thm:deformation} from the Maurer-Cartan element $T$, we have

\begin{pro}\label{pro:danddT}
 Let $T:V\longrightarrow\g$ be an $\huaO$-operator on a Lie algebra $\g$ with respect to a representation $(V;\rho)$. Then we have
 $$
 d_\mrho f=(-1)^kd_Tf,\quad \forall f\in \Hom(\wedge^kV,\g).
 $$
\end{pro}

\noindent \emph{Proof.}
Indeed, for all $u_1,u_2,\cdots,u_{k+1}\in V$ and $f\in \Hom(\wedge^kV,\g)$ , we have
\begin{eqnarray*}
&&(-1)^k(d_Tf)(u_1,u_2,\cdots,u_{k+1})\\ &=&(-1)^{k}\Courant{T,f}(u_1,u_2,\cdots,u_{k+1})\\
&=&\sum_{\sigma\in \mathbb S_{(k,1,0)}}(-1)^{k}(-1)^{\sigma}T(\rho(f(u_{\sigma(1)},\cdots,u_{\sigma(k)}))u_{\sigma(k+1)})\\
&&-\sum_{\sigma\in \mathbb S_{(1,1,k-1)}}(-1)^{\sigma}f(\rho(Tu_{\sigma(1)})u_{\sigma(2)},\cdots,u_{\sigma(k+1)}) +\sum_{\sigma\in \mathbb S_{(1,k)}}(-1)^{\sigma}[Tu_{\sigma(1)},f(u_{\sigma(2)},\cdots,u_{\sigma(k+1)})]\\
&=&\sum_{i=1}^{k+1}(-1)^k(-1)^{k+1-i}T(\rho(f(u_1,\cdots,\hat{u_{i}},\cdots,u_{k+1}))u_{i})\\
&&-\sum_{1\le i<j\le k+1}(-1)^{i-1+j-2} f(\rho(Tu_{i})u_{j},u_1,\cdots,\hat{u_{i}},\cdots,\hat{u_{j}}, \cdots,u_{k+1})\\
&&-\sum_{1\le j<i\le k+1}(-1)^{j-1+i-1} f(\rho(Tu_{i})u_{j},u_1,\cdots,\hat{u_{j}},\cdots,\hat{u_{i}}, \cdots,u_{k+1})\\
&&+\sum_{i=1}^{k+1}(-1)^{i-1}[Tu_{i},f(u_{1},\cdots,\hat{u_i},\cdots,u_{k+1})] \\
&=&\sum_{i=1}^{k+1}(-1)^{i+1}T(\rho(f(u_1,\cdots,\hat{u_{i}},\cdots,u_{k+1}))u_{i}) +\sum_{i=1}^{k+1}(-1)^{i+1}[Tu_{i},f(u_{1},\cdots,\hat{u_i},\cdots,u_{k+1})]\\
&&+\sum_{1\le i<j\le k+1}(-1)^{i+j} f(\rho(Tu_i)(u_j)-\rho(Tu_j)(u_i),u_1,\cdots,\hat{u_i}, \cdots,\hat{u_j},\cdots, u_{k+1})\\
&=&(d_\mrho f)(u_1,u_2,\cdots,u_{k+1}).
\hspace{10.5cm} \qed
\end{eqnarray*}

\begin{defi}
Let $T$ be an $\huaO$-operator on a Lie algebra $\g$ with respect to a representation $(V;\rho)$. Denote by $(\huaC^*(V,\g)=\oplus _{k=0}^{\dim(V)}\huaC^k(V,\g),d_\mrho)$ the above cochain complex. Denote the set of $k$-cocycles by $\huaZ^k(V,\g)$ and the set of $k$-coboundaries by $\huaB^k(V,\g)$. Denote by
  \begin{equation}
  \huaH^k(V,\g)=\huaZ^k(V,\g)/\huaB^k(V,\g), \quad k \geq 0,
  \label{eq:ocoh}
  \end{equation}
the $k$-th cohomology group, called the {\bf $k$-th cohomology group for the \oop $T$}.
\label{de:opcoh}
\end{defi}

We will use these cohomology groups to characterize infinitesimal
and formal deformations of \oops in later sections. See
Theorems \ref{thm:iso3} and \ref{thm:extendable} in particular.
For now, we relate these cohomology
groups to the cohomology groups of the pre-Lie algebra
$(V,\cdot_T)$ obtained from the \oop $T$ by Theorem~\ref{thm:opL}.

\begin{defi}
Let $(V,\cdot_V)$ be a pre-Lie algebra and $W$  a vector space. A
{\bf representation} of $V$ on $W$ is a triple $(W,\rho,\mu)$,
where $\rho:V\longrightarrow \gl(W)$ is a representation of the
sub-adjacent Lie algebra $V^c$ on $W $ and
$\mu:V\longrightarrow \gl(W)$ is a linear map satisfying
\vspace{-.1cm}
\begin{eqnarray}
 \rho(x)\mu(y)u-\mu(y)\rho(x)u=\mu(x\cdot_V y)u-\mu(y)\mu(x)u, \quad \forall~x,y\in V,~ u\in W.
\mlabel{eq:repcond2}
\end{eqnarray}
\end{defi}
\vspace{-.2cm}
Define
\vspace{-.2cm}
$$L:V\longrightarrow \gl(V), \ x\mapsto L_x, \ L_xy=x\cdot_V y;\quad  R:V\longrightarrow \gl(V),\ x\mapsto R_x,\ R_xy=y\cdot_V x, \quad \forall x,y\in V.
\vspace{-.1cm}
$$
Then $(V;L,R)$ is a representation of $V$, called the {\bf regular
representation} of $V$. Note that $L$ also gives a representation
of the sub-adjacent Lie algebra $V^c$ on $V$. See \cite{Bu,Dz} for more details.

The cohomology complex for a pre-Lie algebra $(V,\cdot_V)$ with a representation $(W;\rho,\mu)$ is given as follows \cite{Dz}.
The set of $n$-cochains is given by
$\Hom(\wedge^{n-1}V\otimes V,W),\
n\geq 1.$  The coboundary operator $\dM:\Hom(\wedge^{n-1}V\otimes V,W)\longrightarrow \Hom(\wedge^{n}V\otimes V,W)$ is given by
\vspace{-.3cm}
 \begin{eqnarray}
 \nonumber(\dM f)(x_1, \cdots,x_{n+1})
 \nonumber&=&\sum_{i=1}^{n}(-1)^{i+1}\rho(x_i)f(x_1, \cdots,\hat{x_i},\cdots,x_{n+1})\\
\mlabel{eq:cobold} &&+\sum_{i=1}^{n}(-1)^{i+1}\mu(x_{n+1})f(x_1, \cdots,\hat{x_i},\cdots,x_n,x_i)\\
 \nonumber&&-\sum_{i=1}^{n}(-1)^{i+1}f(x_1, \cdots,\hat{x_i},\cdots,x_n,x_i\cdot_V x_{n+1})\\
\nonumber &&+\sum_{1\leq i<j\leq n}(-1)^{i+j}f([x_i,x_j]^c,x_1,\cdots,\hat{x_i},\cdots,\hat{x_j},\cdots,x_{n+1}),
\mlabel{eq:pLcoh}
\vspace{-.1cm}
\end{eqnarray}
for all $f\in \Hom(\wedge^{n-1}V\otimes V,W)$, $x_i\in V,~i=1,\cdots,n+1$. We use   $\dr$ to denote the coboundary
operator  associated to the regular representation and we obtain the cochain complex \vspace{-.1cm}
$$\CV^*(V,V):=\oplus_{n=1}^{\dim(V)+1} \Hom(\wedge^{n-1}V\otimes V,V).
\vspace{-.1cm}
$$
We denote the corresponding $n$-th cohomology group by
$H_{\rm reg}^n(V,V)$ and
$H_{\rm reg}(V,V):=\oplus_{n}H_{\rm reg}^n(V,V).$

\begin{thm}
Let $T$ be an $\huaO$-operator on a Lie algebra $\g$ with respect to a representation $(V;\rho)$. Then $\Phi$, defined in
Eq.~\eqref{eq:phi}, is a homomorphism of cochain complexes from
$(\huaC^*(V,\g),d_\mrho)$ to $(\CV^*(V,V),\dr)$, that is, $\dr
\Phi = \Phi d_\mrho$. Consequently, $\Phi$ induces a homomorphism
$\Phi_*:\huaH^k(V,\g)\longrightarrow H^{k+1}_{\rm reg}(V,V)$
between the corresponding cohomology groups.
\mlabel{thm:morphismcohomology}
\end{thm}

\noindent\emph{Proof.} Indeed, for all $u_1,u_2,\cdots,u_{k+2}\in V$ and $f\in\Hom(\wedge^kV,\g)$ , we have
\begin{eqnarray*}
&&(\dr \Phi(f))(u_1,u_2,\cdots,u_{k+2})\\
&=&\sum_{i=1}^{k+1}(-1)^{i+1}u_i\cdot_T \Phi(f)(u_1, \cdots,\hat{u_i},\cdots,u_{k+1},u_{k+2}) \\ &&+\sum_{i=1}^{k+1}(-1)^{i+1}\Phi(f)(u_1, \cdots,\hat{u_i},\cdots,u_{k+1},u_i)\cdot_T u_{k+2}\\
&&-\sum_{i=1}^{k+1}(-1)^{i+1}\Phi(f)(u_1, \cdots,\hat{u_i},\cdots,u_{k+1},u_i\cdot_T u_{k+2})\\
&&+\sum_{1\leq i<j\leq k+1}(-1)^{i+j}\Phi(f)([u_i,u_j]^c,u_1,\cdots,\hat{u_i},\cdots,\hat{u_j},\cdots,u_{k+1},u_{k+2})\\
&=&\sum_{i=1}^{k+1}(-1)^{i+1}\rho(Tu_i)\big(\rho(f(u_1, \cdots,\hat{u_i},\cdots,u_{k+1}))(u_{k+2})\big)\\
&&+\sum_{i=1}^{k+1}(-1)^{i+1}\rho(T\rho(f(u_1,\cdots,\hat{u_i},\cdots,u_{k+1}))(u_i))(u_{k+2})\\
&&-\sum_{i=1}^{k+1}(-1)^{i+1}\rho(f(u_1, \cdots,\hat{u_i},\cdots,u_{k+1}))(\rho(Tu_i)(u_{k+2}))\\
&&+\sum_{1\leq i<j\leq k+1}(-1)^{i+j}\rho(f([u_i,u_j]^c,u_1,\cdots,\hat{u_i},\cdots,\hat{u_j},\cdots,u_{k+1}))(u_{k+2})\\
&=&\sum_{i=1}^{k+1}(-1)^{i+1}\rho([Tu_i,f(u_1, \cdots,\hat{u_i},\cdots,u_{k+1})])(u_{k+2})\\
&&+\sum_{i=1}^{k+1}(-1)^{i+1}\rho(T\rho(f(u_1,\cdots,\hat{u_i},\cdots,u_{k+1}))(u_i))(u_{k+2})\\
&&+\sum_{1\leq i<j\leq k+1}(-1)^{i+j}\rho(f([u_i,u_j]^c,u_1,\cdots,\hat{u_i},\cdots,\hat{u_j},\cdots,u_{k+1}))(u_{k+2})\\
&=&\Phi(d_\mrho f)(u_1,u_2,\cdots,u_{k+2}). \hspace{10cm} \qed
\end{eqnarray*}

\begin{rmk}
If $\rho$ is faithful, then $\Phi$ is injective, realizing
the complex $(\huaC^*(V,\g),d_\mrho)$ as a subcomplex of $(\CV^*(V,V),\dr)$.
\end{rmk}

\section{Infinitesimal deformations of an $\huaO$-operator}
\mlabel{sec:infdef}

In  this section, we study infinitesimal deformations of an
$\huaO$-operator using the cohomology  theory  given in the
previous section. In particular, we introduce the notion of a
Nijenhuis element associated to  an $\huaO$-operator, which gives
rise to a trivial infinitesimal deformation of the
$\huaO$-operator. Their relationship with the infinitesimal deformations of
the associated pre-Lie algebra is also studied.

By Remark~\ref{rmk:O-pre}, there is a close relationship between
the set of \oops and the set of pre-Lie algebras as Maurer-Cartan
elements in the respective graded Lie algebras. Hence it is
natural to consider the relationships between two \oops in terms
of the ones between the corresponding pre-Lie algebras. On the
other hand, the classification of pre-Lie
algebras in the sense of isomorphism is interpreted as the
classification of bijective 1-cocycles in the sense of equivalence
(\cite{Bai2}) or the classification of \'etale affine
representations in the sense of equivalence (\cite{Bau}).
Motivated by these two types of equivalences, we give

    \begin{defi}
      Let $T$ and $T'$ be $\huaO$-operators on a Lie algebra $\g$ with respect to a representation $(V;\rho)$. A {\bf homomorphism} from $T'$ to $T$ consists of a Lie algebra homomorphism  $\phi_\g:\g\longrightarrow\g$ and a linear map $\phi_V:V\longrightarrow V$ such that
      \begin{eqnarray}
        T\circ \phi_V&=&\phi_\g\circ T',\mlabel{defi:isocon1}\\
        \phi_V\rho(x)(u)&=&\rho(\phi_\g(x))(\phi_V(u)),\quad\forall x\in\g, u\in V.\mlabel{defi:isocon2}
      \end{eqnarray}
      In particular, if both $\phi_\g$ and $\phi_V$ are  invertible,  $(\phi_\g,\phi_V)$ is called an  {\bf isomorphism}  from $T'$ to $T$.
    \mlabel{defi:isoO}
    \end{defi}

We refer the reader to~\cite{CP} for a weaker version of isomorphism up to a scalar.

\begin{pro}
Let $T$ and $T'$ be two $\huaO$-operators on a Lie algebra $\g$ with respect to a representation $(V;\rho)$ and $(\phi_\g,\phi_V)$   a homomorphism (resp. an isomorphism) from $T'$ to $T$. Then $\phi_V$ is a homomorphism (resp. an isomorphism) of pre-Lie algebras from $(V,\cdot_{T'})$ to $(V,\cdot_{T})$.
\end{pro}
\noindent \emph{Proof.} This is because, for all $u,v\in V$, we
have
\begin{eqnarray*}
\phi_V(u\cdot_{T'} v)&=&\phi_V
\rho(T'u)(v)=\rho(\phi_\g(T'u))(\phi_V(v))=\rho(T(\phi_V(u)))(\phi_V(v))
=\phi_V(u)\cdot_T \phi_V(v). \hspace{.5cm} \qed
\end{eqnarray*}

    \begin{defi}
    Let $T$   be an $\huaO$-operator  on a Lie algebra $\g$ with respect to a representation $(V;\rho)$ and $\frkT:V\longrightarrow\g$ a linear map. If  $T_t=T+t\frkT$ is still an $\huaO$-operator  on the Lie algebra $\g$ with respect to the representation $(V;\rho)$ for all $t$, we say that $\frkT$ generates a {\bf one-parameter infinitesimal deformation} of the $\huaO$-operator $T$.
    \end{defi}
It is direct to check that $T_t=T+t\frkT$ is a one-parameter
infinitesimal deformation of an \oop $T$ if and only if
 for any $u,v\in V$,
\vspace{-.1cm}
\begin{eqnarray}
~[Tu,\frkT v]+[\frkT u,Tv]&=&T(\rho(\frkT u)(v)-\rho(\frkT v)(u))+\frkT(\rho(Tu)(v)-\rho(Tv)(u)),\mlabel{eq:deform1}\\
~[\frkT u,\frkT v]&=&\frkT(\rho(\frkT u)(v)-\rho(\frkT v)(u)).
\mlabel{eq:deform2}
\vspace{-.1cm}
\end{eqnarray}
Note that Eq.~\eqref{eq:deform1} means that $\frkT$ is a 1-cocycle of the sub-adjacent Lie algebra $(V,[\cdot,\cdot]_T)$ with coefficients in $\g$ and Eq.~\eqref{eq:deform2} means that $\frkT$ is an $\huaO$-operator on the Lie algebra $\g$ associated to the representation $(V;\rho)$.

Now turning to a pre-Lie algebra $(V,\cdot_V)$, let  $\omega:\otimes^2V\longrightarrow V$ be a linear map. If for any $t\in\K$, the multiplication $\cdot_t$ defined by
\vspace{-.3cm}
$$
u\cdot_tv:=u\cdot_V v+t\omega(u,v), \;\forall u,v\in V,
\vspace{-.2cm}
$$
also gives a pre-Lie algebra structure, we say that $\omega$ generates a {\bf one-parameter infinitesimal deformation} of the pre-Lie algebra $(V,\cdot_V)$.

The two types of infinitesimal deformations are related as follows.

\begin{pro}
 If $\frkT$ generates a one-parameter infinitesimal deformation of an $\huaO$-operator $T$ on a Lie algebra $\g$ with respect to a representation $(V;\rho)$, then the product $\omega_\frkT$ on $V$ defined by
\vspace{-.2cm}
   $$
   \omega_\frkT(u,v):=\rho(\frkT u)(v),\quad\forall u,v\in V,
   $$
generates a one-parameter infinitesimal deformation of the associated pre-Lie algebra $(V,\cdot_T)$.
\end{pro}

\begin{proof} Denote by $\cdot_t$ the corresponding pre-Lie algebra structure associated to the $\huaO$-operator $T+t\frkT$. Then we have
\vspace{-.1cm}
$$
u\cdot_t v=\rho((T+t\frkT)(u))(v)=\rho(Tu)(v)+t\rho(\frkT
u)(v)=u\cdot_T v+t\omega_\frkT (u,v), \forall u,v\in V,
\vspace{-.1cm}
$$
which implies that $\omega_\frkT$ generates a one-parameter infinitesimal deformation of $(V,\cdot_T)$.\end{proof}

\begin{cor}
 If $\frkT$ generates a one-parameter infinitesimal deformation of an $\huaO$-operator $T$ on a Lie algebra $\g$ with respect to a representation $(V;\rho)$, then the product $ \varpi_\frkT$ on $V$ defined by
\vspace{-.1cm}
   $$
    \varpi_\frkT(u,v):=\rho(\frkT u)(v)-\rho(\frkT v)(u),\quad\forall u,v\in V,
\vspace{-.1cm}
   $$
generates a one-parameter infinitesimal
deformation of the sub-adjacent Lie algebra $(V,[\cdot,\cdot]_T)$
of the associated pre-Lie algebra $(V,\cdot_T)$.
\end{cor}
\vspace{-.1cm}

\begin{defi} Let $T$ be an \oop on a Lie algebra $\g$
with respect to a representation $(V;\rho)$. Two one-parameter
infinitesimal deformations $T^1_t=T+t\frkT_1$ and
$T^2_t=T+t\frkT_2$ are said to be {\bf equivalent} if there exists
an $x\in\g$ such that $({\Id}_\g+t\ad_x,{\Id}_V+t\rho(x))$ is a
homomorphism   from $T^2_t$ to $T^1_t$. In particular, a
one-parameter infinitesimal deformation $T_t=T+t\frkT$ of an
$\huaO$-operator $T$ is said to be {\bf trivial} if there exists
an $x\in\g$ such that $({\Id}_\g+t\ad_x,{\Id}_V+t\rho(x))$ is a
homomorphism   from $T_t$ to $T$.
\end{defi}

Let $({\Id}_\g+t\ad_x,{\Id}_V+t\rho(x))$ be a homomorphism from
$T^2_t$ to $T^1_t$. Then ${\Id}_\g+t\ad_x$ is a Lie algebra
endomorphism of $\g$. Thus, we have
$$
({\Id}_\g+t\ad_x)[y,z]=[({\Id}_\g+t\ad_x)(y),({\Id}_\g+t\ad_x)(z)], \;\forall y,z\in \g,
$$
which implies that $x$ satisfies
\begin{equation}
 [[x,y],[x,z]]=0,\quad \forall y,z\in\g.
 \mlabel{eq:Nij1}
 \end{equation}
Then by Eq.~\eqref{defi:isocon1}, we get
$$
(T+t\frkT_1)({\Id}_V+t\rho(x))(u)=({\Id}_\g+t\ad_x)(T+t\frkT_2)(u),\quad\forall u\in V,
$$
which implies
\begin{eqnarray}
 (\frkT_2-\frkT_1)(u)&=&T\rho(x)(u)+[Tu,x],\mlabel{eq:deforiso1} \\
  \frkT_1\rho(x)(u)&=&[x,\frkT_2u], \; \forall u\in V.
  \mlabel{eq:deforiso2}
\end{eqnarray}
Finally, Eq.~\eqref{defi:isocon2} gives
$$
({\Id}_V+t\rho(x))\rho(y)(u)=\rho(({\Id}_\g+t\ad_x)(y))({\Id}_V+t\rho(x))(u),\quad \forall y\in\g, u\in V,
$$
which implies that $x$ satisfies
\begin{equation}
  \rho([x,y])\rho(x)=0,\quad\forall y\in\g.\mlabel{eq:Nij2}
\end{equation}

Note that Eq.~\eqref{eq:deforiso1} means that $\frkT_2-\frkT_1=d_\mrho x$ for $\mrho$ defined in Lemma~\ref{lem:rep}. Thus, we have

\begin{thm}\label{thm:iso3} Let $T$ be an \oop on a Lie algebra $\g$
with respect to a representation $(V;\rho)$.
  If two one-parameter infinitesimal deformations $T^1_t=T+t\frkT_1$ and $T^2_t=T+t\frkT_2$ are equivalent, then $\frkT_1$ and $\frkT_2$ are in the same cohomology class of $\huaH^1(V,\g)=\huaZ^1(V,\g)/\huaB^1(V,\g)$ defined in
  Definition~\ref{de:opcoh}.
\end{thm}

    \begin{defi}
Let $T$ be an \oop on a Lie algebra $\g$ with respect to a representation $(V;\rho)$. An element $x\in\g$ is called a {\bf Nijenhuis element} associated to $T$ if $x$ satisfies Eqs.~\eqref{eq:Nij1}, \eqref{eq:Nij2} and the equation
      \begin{eqnarray}
        ~[x,[Tu,x]+T\rho(x)(u)]=0,\quad \forall u\in V.
         \mlabel{eq:Nij3}
        \end{eqnarray}
   Denote by $\Nij(T)$ the set of Nijenhuis elements associated to an $\huaO$-operator $T$.
    \end{defi}

By Eqs.~\eqref{eq:Nij1}-\eqref{eq:Nij2}, it is obvious that a
trivial one-parameter infinitesimal deformation gives rise to a
Nijenhuis element. The following result is in close analogue to
the fact that the differential of a Nijenhuis operator on a Lie
algebra generates a trivial one-parameter infinitesimal
deformation of the Lie algebra~\cite{Do}, justifying the notion of
Nijenhuis elements.

  \begin{thm}\label{thm:trivial}
   Let $T$ be an $\huaO$-operator on a Lie algebra $\g$ with respect to a representation $(V;\rho)$. Then for any  $x\in \Nij(T)$, $T_t:=T+t \frkT$ with $\frkT:=d_\mrho x$ is a trivial one-parameter infinitesimal  deformation of the $\huaO$-operator $T$.
\end{thm}
\begin{proof}
First $\frkT$ is closed since $\frkT=d_\mrho x$.
To show that $\frkT=d_\mrho x$ generates a trivial one-parameter infinitesimal deformation of
the $\huaO$-operator $T$, we only need to verify that
Eq.~\eqref{eq:deform2} holds. By Eq.~\eqref{eq:Nij1}, we have,
for any $u,v\in V$,
\begin{eqnarray*}
      &&[\frkT u,\frkT v]-\frkT(\rho(\frkT u)(v)-\rho(\frkT v)(u))\\
      &=& [[Tu,x],[Tv,x]]+[[Tu,x],T\rho(x)(v)]+[T\rho(x)(u),[Tv,x]]+[T\rho(x)(u),T\rho(x)(v)]\\
      &&-[T\rho([Tu,x])(v),x]-[T\rho(T\rho(x)(u))(v),x]+[T\rho([Tv,x])(u),x]+[T\rho(T\rho(x)(v))(u),x]\\
      &&-T\rho(x)\rho([Tu,x])(v)-T\rho(x)\rho(T\rho(x)(u))(v)+T\rho(x)\rho([Tv,x])(u)+T\rho(x)\rho(T\rho(x)(v))(u)\\
      &=& [[Tu,x],T\rho(x)(v)]+[T\rho(x)(u),[Tv,x]]+\underline{T\rho(T\rho(x)(u))\rho(x)(v)}\underbrace{-T\rho(T\rho(x)(v))\rho(x)(u)}\\
      &&-[T\rho([Tu,x])(v),x]-[T\rho(T\rho(x)(u))(v),x]+[T\rho([Tv,x])(u),x]+[T\rho(T\rho(x)(v))(u),x]\\
      &&\underline{-T\rho(x)\rho([Tu,x])(v)-T\rho(x)\rho(T\rho(x)(u))(v)}\underbrace{+T\rho(x)\rho([Tv,x])(u)+T\rho(x)\rho(T\rho(x)(v))(u)}.
    \end{eqnarray*}
   By Eqs.~\eqref{eq:Nij2} and \eqref{eq:Nij3}, the under-braced terms add to zero.
Similarly, the underlined terms add to zero.
For the other terms, by Eqs.~\eqref{eq:Nij1} and \eqref{eq:Nij3}, we have
\begin{eqnarray*}
  &&[[Tu,x],T\rho(x)(v)]-[T\rho([Tu,x])(v),x]+[T\rho(T\rho(x)(v))(u),x]\\
  &=&[Tu,[x,T\rho(x)(v)]]+[[Tu,T\rho(x)(v)],x]-[T\rho([Tu,x])(v),x]+[T\rho(T\rho(x)(v))(u),x]\\
  &=&-[Tu,[x,[Tv,x]]]+[T\rho(Tu)\rho(x)(v),x]-[T\rho([Tu,x])(v),x]\\
  &=&-[x,[Tu,[Tv,x]]]+[T\rho(x)\rho(Tu)(v),x]\\
  &=&-[x,[Tu,[Tv,x]]]+[x,[T\rho(Tu)(v),x]].
\end{eqnarray*}
Similarly, we have
\begin{eqnarray*}
&&[T\rho(x)(u),[Tv,x]]-[T\rho(T\rho(x)(u))(v),x]+[T\rho([Tv,x])(u),x]\\
 &=&[x,[Tv,[Tu,x]]]-[x,[T\rho(Tv)(u),x]].
\end{eqnarray*}
Therefore,
 \begin{eqnarray*}
      &&[\frkT u,\frkT v]-\frkT(\rho(\frkT u)(v)-\rho(\frkT v)(u))\\
&=&-[x,[Tu,[Tv,x]]]+[x,[Tv,[Tu,x]]]+[x,[T\rho(Tu)(v),x]]-[x,[T\rho(Tv)(u),x]]\\
&=&-[x,[[Tu,Tv],x]]+[x,[[Tu,Tv],x]]=0,
\end{eqnarray*}
which means that $\frkT:=d_\mrho x$ generates  a   one-parameter infinitesimal  deformation of $T$.

Further, since $x$ is a Nijenhuis element, it is straightforward
to deduce that $({\Id}_\g+t\ad_x,{\Id}_V+t\rho(x))$ gives the
desired homomorphism between $T_t$ and $T$. Thus, the deformation is trivial.
\end{proof}

Now we recall the notion of a Nijenhuis operator on a pre-Lie algebra   given in \cite{WBLS}, which gives rise to a trivial one-parameter infinitesimal deformation of a pre-Lie algebra.

\begin{defi}
  A linear map $N:V\longrightarrow V$ on a pre-Lie algebra $(V,\cdot_V) $ is called a {\bf Nijenhuis operator} if
  \begin{equation}
    (Nu)\cdot_V (Nv)=N((Nu)\cdot_V v+u\cdot_V (Nv)-N(u\cdot_V v)),\quad\forall u,v\in V.
  \end{equation}
\end{defi}

For its connection with a Nijenhuis element associated to
an \oop, we   have

\begin{pro}
Let $x\in\g$ be a Nijenhuis element associated to an
$\huaO$-operator $T$ on a Lie algebra $\g$ with respect to a
representation $(V;\rho)$. Then $\rho(x)$ is a Nijenhuis operator
on the associated pre-Lie algebra $(V,\cdot_T)$.
    \end{pro}
\noindent \emph{Proof.} For the proof, we just need to check, by
Eq.~\eqref{eq:Nij2}, for all $u,v\in V$,
    \begin{eqnarray*}
     && \rho(x)(\rho(x)(u)\cdot_T v+u\cdot_T\rho(x)(v)-\rho(x)(u\cdot_Tv))-\rho(x)(u)\cdot_T\rho(x)(v)\\
     &=&\rho(x)\Big(\rho(T\rho(x)(u))(v)+\rho(Tu)\rho(x)(v)-\rho(x)\rho(Tu)(v)\Big)-\rho(T\rho(x)(u))\rho(x)(v)\\
     &=&[\rho(x),\rho(T\rho(x)(u))]+[\rho(x),\rho([Tu,x])](v)\\
     &=&\rho([x,T\rho(x)(u)+[Tu,x]])(v)
     =0. \hspace{7cm} \qed
    \end{eqnarray*}

\section{Formal deformations of an $\huaO$-operator}
\mlabel{sec:fordef}

In this section, first we study one-parameter formal deformations
of an $\huaO$-operator and consider the rigidity of the
$\huaO$-operator. Then we study order $n$ deformations of an
$\huaO$-operator.   We show that the obstruction of an order $n$ deformation being extendable is given by a class in the second
cohomology group.

Let $\K[[t]]$ be the ring of power series in one variable
$t$, taking as the inverse limit of the system $(\bfk[t]/(t^n),
\pi_{n+1,n})$, where $\pi_{n+1,n}:\bfk[t]/(t^{n+1}) \to
\bfk[t]/(t^n)$ is the natural quotient map. For any $\K$-linear space $V$, we let
$V[[t]]$ denote the inverse limit of the system $(V\ot
(\bfk[t]/(t^n)),{\rm Id}_V\ot \pi_{n+1,n})$, formally regarded as
power series in $t$ with coefficients in $V$.

If in addition, $\g$ is a Lie algebra over $\K$, then
$\g[[t]]$ is a Lie algebra over $\K$ by
\begin{equation}
\bigg[\sum_{i\geq 0} a_it^i,\sum_{j\geq 0}b_it^j\bigg]:=\sum_{k\geq
0}\sum_{i+j=k}[a_i,b_j]t^k,\quad\forall  a_i,b_j\in \g.
\end{equation}
For any representation $(V;\rho)$ of $\g$, there is a natural
action of $\g[[t]]$ on $V[[t]]$ induced by $\rho$ in an obvious
way  which is still denoted by
$\rho$. In fact, since $\rho:\g\rightarrow {\rm
Hom_{\K}}(V,V)$ is $\K$-linear, it can be extended
to be a $\K[[t]]$-module map from $\g[[t]]$ to ${\rm
Hom}_{\K[[t]]}(V[[t]],V[[t]])$.

Let $T$   be an $\huaO$-operator  on a Lie algebra $\g$ with respect to a representation $(V;\rho)$.
Consider a $t$-parameterized family of linear operations
\begin{eqnarray}
T_t=\sum_{i\geq0}\tau_i t^i,\quad \tau_i\in \Hom_{\K}(V,\g),
\label{eq:tdeform}
\end{eqnarray}
that is, $T_t\in{\rm Hom}_{\K}(V,\g)[[t]]\subset {\rm
Hom}_{\K}(V,\g[[t]])$.  Extend it to be a $\K[[t]]$-module map from $V[[t]]$ to $\g[[t]]$ which is still denoted by $T_t$.

\begin{defi}\label{defi:dO}
 If  $T_t=\sum_{i\geq0}\tau_i t^i$ with $\tau_0=T$ satisfies
\begin{eqnarray}
[T_t(u),T_t(v)]=T_t\Big(\rho(T_t(u))(v)-\rho(T_t(v))(u)\Big),\;\;\forall
u,v\in V, \mlabel{O-operator}
\end{eqnarray}
we say that $T_t$ is a {\bf  one-parameter formal deformation} of
the $\huaO$-operator $T.$
\end{defi}

\begin{rmk} The left hand side of Eq.~(\ref{O-operator}) holds in
the Lie algebra $\g[[t]]$, whereas the right hand side makes sense
since $T_t$ is a $\K[[t]]$-module map.
\end{rmk}

Recall~\cite{Bu0} that a one-parameter formal deformation of a
pre-Lie algebra $(A,\cdot)$ is a power series
$f_t=\sum_{i=0}^\infty \alpha_i t^i$ such that $\alpha_0(a,b)=a\cdot b$
 for any $a,b\in A$  and $f_t$ defines a pre-Lie algebra
structure on $A[[t]]$.

Building on the relationship between \oops and pre-Lie algebras, we have
\begin{pro}
 If $T_t=\sum_{i\geq0}\tau_i t^i$ is a one-parameter formal deformation of an $\huaO$-operator $T$ on a Lie algebra $\g$ with respect to a representation $(V;\rho)$, then $\cdot_{T_t}$ defined by
   $$
   u\cdot_{T_t}v=\sum_{i\geq 0}\rho(\tau_iu)(v) t^i,\quad\forall u,v\in V,
   $$
is a one-parameter formal deformation of the associated pre-Lie algebra $(V,\cdot_T)$.
\end{pro}

Applying Eq.~\eqref{eq:tdeform} to expand Eq.~\eqref{O-operator} and collecting coefficients of $t^n$, we see that Eq.~\eqref{O-operator} is equivalent to the system of equations
\begin{eqnarray}
\sum\limits_{i+j=k\atop
i,j\geq0}\Big([\tau_i(u),\tau_j(v)]-\tau_i(\rho(\tau_j(u))(v)-\rho(\tau_j(v))(u))\Big)=0,
\;\;\forall k\geq 0, u,v\in V. \mlabel{deformation-equation}
\end{eqnarray}

\begin{pro}
Let $T_t=\sum_{i\geq0}\tau_i t^i$ be a one-parameter formal
deformation of an $\huaO$-operator $T$ on a Lie algebra $\g$ with
respect to a representation $(V;\rho)$. Then $\tau_1$ is a
$1$-cocycle on the Lie algebra $V^c=(V,[\;,\;]_T)$ with coefficients in $\g$, that is, $d_\mrho \tau_1=0$. 
\mlabel{pro:cocycle}
\end{pro}
\begin{proof} 
For $k=1$, Eq.~\eqref{deformation-equation} is equivalent to
  \vspace{-.1cm}
$$ [Tu,\tau_1(v)]-[Tv,\tau_1(u)]-T(\rho(\tau_1(u))(v)-\rho(\tau_1(v))(u))-\tau_1(\rho(Tu)(v)-\rho(Tv)(u))=0,\;\;\forall u,v\in V.
  \vspace{-.1cm}
$$
Thus, $\tau_1$ is a $1$-cocycle.
\end{proof}

\begin{defi}
Let $T$ be an \oop on a Lie algebra $\g$ with respect to a
representation $(V;\rho)$. The $1$-cocycle $\tau_1$ in
Proposition~\ref{pro:cocycle} is called the {\bf infinitesimal} of
the one-parameter formal deformation $T_t=\sum_{i\geq0}\tau_i t^i$ of $T$.
  \vspace{-.1cm}
\end{defi}

\begin{defi}
 Two formal deformations $\overline{T}_t =\sum_{i\geq0}\bar{\tau}_i t^i$ and $T_t=\sum_{i\geq0}\tau_i
 t^i$ of an $\mathcal O$-operator $T={\bar {\tau}}_0=\tau_0$
on a Lie algebra $\g$ with respect to a representation $(V;\rho)$
are said to be {\bf equivalent} if there exist  $x\in\g$,
$\phi_i\in\gl(\g)$ and $\varphi_i\in\gl(V)$, $i\geq 2$, such that
for
\begin{equation}
\phi_t:={\Id}_\g+t\ad_x+\sum_{i\ge2}\phi_it^i,\;\;\varphi_t :={\Id}_V+t\rho(x)+\sum_{i\ge2}\varphi_it^i,
\label{eq:phi5}
\vspace{-.4cm}
\end{equation}
the following conditions hold:
\begin{enumerate}
\item[\rm(i)] $[\phi_t(x),\phi_t(y)]=\phi_t[x,y]$ for all
$x,y\in \g;$
\item[\rm(ii)]   $T_t\circ
\varphi_t=\phi_t\circ \overline{T}_t$ as $\K[[t]]$-module maps;
\item[\rm(iii)]
$\varphi_t\rho(x)u=\rho(\phi_t(x))\varphi_t(u)$ for all $x\in\g,
u\in V$.
\end{enumerate}
In particular, a formal deformation $T_t$ of an $\huaO$-operator
$T$ is said to be {\bf trivial} if there exists an $x\in\g$,
$\phi_i\in\gl(\g)$ and $\varphi_i\in\gl(V)$, $i\geq 2$, such that
$(\phi_t,\varphi_t)$ defined by Eq.~(\ref{eq:phi5}) gives an
equivalence  between $T_t$ and $T$, with the latter regarded as a
deformation of itself.
\end{defi}

\begin{thm}
If two one-parameter formal deformations of an $\huaO$-operator
on a Lie algebra $\g$ with respect to a representation
$(V;\rho)$ are equivalent, then their infinitesimals are in the
same cohomology class.
\end{thm}

\begin{proof} Let $(\phi_t,\varphi_t)$ be the two maps defined by
Eq.~(\ref{eq:phi5}) which gives an equivalence between two
deformations $\overline{T}_t=\sum_{i\geq0}\overline{\tau}_i t^i$ and
$T_t=\sum_{i\geq0}\tau_i t^i$ of an \oop $T$. By $\phi_t \circ\overline{T}_t=
T_t\circ\varphi_t$, we have
\begin{eqnarray*}
\bar{\tau}_1(v)=\tau_1(v)+T\rho(x)(v)-[x,Tv]=\tau_1(v)+(d_\mrho
x)(v),\quad \forall v\in V,
\end{eqnarray*}
which implies that $\bar{\tau}_1$ and $\tau_1$ are in the same
cohomology class.   \end{proof}

\begin{defi}
An $\huaO$-operator $T$ is {\bf rigid} if all one-parameter formal deformations of $T$ are trivial.
\end{defi}

As a cohomological condition of the rigidity, we have the following result which suggests that the rigidity of an \oop is a very strong condition.

\begin{pro} Let $T:V\rightarrow \g$ be an \oop on a Lie algebra $\g$ with respect to a representation $(V;\rho)$. If $\huaZ^1(V,\g)=d_\mrho(\Nij(T))$, then  $T$ is rigid.
\end{pro}
\begin{proof} Let
$T_t=\sum_{i\geq0}\tau_i t^i$ be a one-parameter formal
deformation of the $\huaO$-operator $T$. Then Proposition
\ref{pro:cocycle} gives $\tau_1\in \huaZ^1(V,\g)$. By the
assumption,  $\tau_1=-d_\mrho x$ for some $x\in\Nij(T)$. Then
setting $\phi_t={\Id_\g}+t\ad_x$ and $\varphi_t={\Id}_V+t\rho(x)$,
we get a formal deformation $\overline{T}_t:=\phi_t^{-1}\circ
T_t\circ\varphi_t.$ Thus, $\overline{T}_t$ is equivalent to $T_t$.
Moreover, we have
\begin{eqnarray*}
\overline{T}_t(v)&=&({\Id}-\ad_xt+\ad^2_xt^2+\cdots+(-1)^i\ad^i_xt^{i}+\cdots)(T_t(v+\rho(x)(v)t))\\
            &=&T(v)+(\tau_1(v)+T\rho(x)(v)-[x,Tv])t+\bar{\tau}_2(v)t^2+\cdots\\
            &=&T(v)+\bar{\tau}_2(v)t^2+\cdots.
\end{eqnarray*}
Repeating this procedure,  we find that $T_t$ is equivalent to $T$. \end{proof}

We next study the obstruction of a deformation of order $n$ from being extendable.

\begin{defi}Let $T:V\longrightarrow\g$ be an $\huaO$-operator on a Lie algebra $\g$ with respect to a representation $(V;\rho)$.
If $T_t=\sum_{i=0}^n\tau_i t^i$ with $\tau_0=T$,
$\tau_i:V\rightarrow \g$, $i=2,\cdots, n$, defines a
$\K[[t]]/(t^{n+1})$-module map from $V[[t]]/(t^{n+1})$ to the Lie
algebra $\g[[t]]/(t^{n+1})$  satisfying
\begin{eqnarray}
[T_t(u),T_t(v)]=T_t\Big(\rho(T_t(u))(v)-\rho(T_t(v))(u)\Big),\;\;\forall
u,v\in V, \mlabel{O-operator of order n}
\end{eqnarray} we say that $T_t$
is an {\bf  order $n$ deformation} of the $\huaO$-operator $T$.
\end{defi}

\begin{rmk} Obviously, the left hand side of Eq.~(\ref{O-operator of order
n}) holds in the Lie algebra $\g[[t]]/(t^{n+1})$ and the  right  hand
side makes sense since $T_t$ is a $\K[[t]]/(t^{n+1})$-module
map.
\end{rmk}

\begin{defi}
Let $T_t=\sum_{i=0}^n\tau_i t^i$ be an order $n$ deformation  of
an $\huaO$-operator $T$ on a Lie
algebra $\g$ with respect to a representation $(V;\rho)$.
 If there
exists a $1$-cochain $\tau_{n+1}\in \huaC^1(V,\g)$ such that
$\widetilde{T}_t=T_t+\tau_{n+1}t^{n+1}$ is an order $n+1$
deformation of the $\huaO$-operator $T$, then we say that $T_{t}$
is {\bf extendable}.
\end{defi}

\begin{pro}
Let $T_t=\sum_{i=0}^n\tau_i t^i$ be an order $n$
deformation  of an $\huaO$-operator $T$ on a Lie algebra $\g$ with
respect to a representation $(V;\rho)$. Define $\Ob \in
\huaC^2(V,\g)$ by
\begin{eqnarray}
\Ob(u,v):=\sum\limits_{i+j=n+1\atop
i,j\geq1}\Big([\tau_i(u),\tau_j(v)]-\tau_i(\rho(\tau_j(u))(v)-\rho(\tau_j(v))(u))\Big),\;\;
\forall u,v\in V. \mlabel{ob}
\end{eqnarray}
Then the 2-cochain $\Ob$ is a $2$-cocycle, that is, $d_\mrho
\Ob=0$.
\end{pro}

\noindent \emph{Proof.} By the bracket in Eq.~\eqref{o-bracket},
we have $\Ob=-\frac{1}{2}\sum\limits_{i+j=n+1\atop
i,j\geq1}\Courant{\tau_i,\tau_j}. $ Since $T_t$ is an order
$n$ deformation  of the $\huaO$-operator $T$, for all $0\leq i\leq
n$, we have
\begin{eqnarray}
\sum\limits_{k+l=i\atop
k,l\geq0}\Big([\tau_k(u),\tau_l(v)]-\tau_k(\rho(\tau_l(u))(v)-\rho(\tau_l(v))(u))\Big)=0, \quad\forall u,v\in V,
\mlabel{deformation-1}
\end{eqnarray}
which is equivalent to
\begin{eqnarray}
-\frac{1}{2}\sum\limits_{k+l=i\atop k,l\geq1}\Courant{\tau_k,\tau_l}=\Courant{T,\tau_{i}}.
\mlabel{deformation-2}
\end{eqnarray}
Then we have
\begin{eqnarray*}
d_\mrho \Ob&=&(-1)^2\Courant{T,\Ob}\\
             &=&-\frac{1}{2}\sum\limits_{i+j=n+1\atop i,j\geq1}\Courant{T,\Courant{\tau_i,\tau_j}}\\
             &=&-\frac{1}{2}\sum\limits_{i+j=n+1\atop i,j\geq1}\Big(\Courant{\Courant{T,\tau_i},\tau_j}-\Courant{\tau_i,\Courant{T,\tau_j}}\Big)\\
             &\stackrel{\eqref{deformation-2}}{=}&\frac{1}{4}\sum\limits_{i'+i''+j=n+1\atop i',i'',j\geq1}\Courant{\Courant{\tau_{i'},\tau_{i''}},\tau_j}-\frac{1}{4}
       \sum\limits_{i+j'+j''=n+1\atop i,j',j''\geq1}\Courant{\tau_i,\Courant{\tau_{j'},\tau_{j''}}}\\
             &=&\frac{1}{2}\sum\limits_{i'+i''+j=n+1\atop i',i'',j\geq1}\Courant{\Courant{\tau_{i'},\tau_{i''}},\tau_j}
             =0. \hspace{8cm} \qed
\end{eqnarray*}

 \begin{defi}
  Let $T_t=\sum_{i=0}^n\tau_i t^i$ be an order $n$ deformation  of an $\huaO$-operator $T$ on a Lie algebra
$\g$ with respect to a representation $(V;\rho)$.  The cohomology
class $[\Ob]\in\huaH^2(V,\g)$ is called the {\bf obstruction
class} of  $T_t$ being extendable.
  \vspace{-.1cm}
  \end{defi}

\begin{thm}\label{thm:extendable}
Let $T_t=\sum_{i=0}^n\tau_i t^i$ be an order $n$ deformation of an \oop $T$ on a Lie algebra $\g$ with respect to a representation $(V;\rho)$. Then
$T_t$ is extendable if and only if the obstruction class $[\Ob]$ is trivial.
\end{thm}
\begin{proof}
Suppose that an order $n$ deformation $T_t$   of the $\huaO$-operator $T$ extends to an order $n+1$ deformation. Then Eq. \eqref{deformation-2} holds for $i=n+1$. Thus, we have
$
\Ob=-d_\mrho \tau_{n+1},
$
which implies that the obstruction class $[\Ob]$ is trivial.

Conversely, if the obstruction class $[\Ob]$ is trivial, suppose that
$
\Ob=-d_\mrho \tau_{n+1}
$
for some 1-cochain $\tau_{n+1} \in \huaC^1(V,\g)$. Set
$
\widetilde{T}_t:=T_t+\tau_{n+1}t^{n+1}.
$
Then $\widetilde{T}_t$ satisfies Eq.~\eqref{deformation-1} for $0\leq i\leq n+1$. So $\widetilde{T}_t$ is an order $n+1$ deformation, which means that  $T_t$ is extendable. 
  \vspace{-.1cm}
\end{proof}

\begin{cor} Let $T$ be an \oop on a Lie
algebra $\g$ with respect to a representation $(V;\rho)$. If
$\huaH^2(V,\g)=0$, then every $1$-cocycle in $\huaZ^1(V,\g)$ is
the infinitesimal of some one-parameter formal deformation of the
$\huaO$-operator $T$.
  \vspace{-.2cm}
\end{cor}

\section{Deformations of Rota-Baxter operators of weight 0}
\mlabel{sec:rbar}

In this section we consider Rota-Baxter operators of weight
0 whose definition is recalled in Definition~\ref{de:conc}.
They form an important case of \oops,   for the {\bf adjoint
representation}
$$\ad: \g \longrightarrow \gl (\g),\quad x \mapsto \ad_x=[x,\cdot], \quad \forall x\in \g.$$
The deformation theory in the previous sections specializes to a
deformation theory of  Rota-Baxter operators of weight 0.
We will provide some applications without repeating all the
details.

A Lie algebra $\g$ with a Rota-Baxter operator $R$ of weight 0 is
called a {\bf Rota-Baxter Lie algebra}. The associated pre-Lie algebra structure on $\g$ is given by
$x\cdot_R y:=[Rx,y]$ for all $x,y\in g$, and its sub-adjacent Lie algebra structure is given by $[x,y]_R:=[Rx,y]+[x,Ry]$ for all $x,y \in \g$.

As a consequence of Proposition~\ref{pro:gla} and Theorem~\ref{thm:deformation}, we have

\begin{cor} Let $\g$ be a Lie algebra.
\begin{enumerate}
\item
$\huaC^*(\g,\g):=(\oplus_{k=0}^{\dim(\g)}\Hom(\wedge^k\g,\g),\Courant{\cdot,\cdot})$ is  a graded Lie algebra, where the graded Lie bracket $\Courant{\cdot,\cdot}$ is given by Eq.~\eqref{o-bracket}.
\item
$R\in\gl(\g)$ is a Rota-Baxter operator of weight $0$ if and only if $R$ is a Maurer-Cartan element of $\huaC^*(\g,\g)$.
\item
A Rota-Baxter operator $R$ of weight $0$ on $\g$ gives rise to a differential $d_R$ on $\huaC^*(\g,\g)$ by
$d_R:=\Courant{R,\cdot}.$
Further a linear map $R':\g\longrightarrow \g$, $R+R'$ is a Rota-Baxter operator of weight $0$ if and only if $R'$ is a Maurer-Cartan element of the differential graded Lie algebra $(\huaC^*(\g,\g),\Courant{\cdot,\cdot},d_R)$.
\end{enumerate}
\end{cor}

By Lemma \ref{lem:rep}, we obtain
    \begin{cor}
   Let $R$ be a Rota-Baxter operator  of weight $0$ on a Lie algebra $\g$. Then
      \begin{equation}
\mrho:\g\longrightarrow\gl(\g), \quad      \mrho(x):=  \ad_{R(x)}- R\circ\ad_x, \quad \forall x\in \g,
      \end{equation}
      is a representation of the sub-adjacent Lie algebra $(\g,[\cdot,\cdot]_R)$.
   \mlabel{cor:repR}
    \end{cor}

    \begin{rmk}
      For the sub-adjacent Lie algebra $(\g,[\cdot,\cdot]_R)$, there are already two representations on itself. The first one is the adjoint representation $\rho_1$ given by
      $$
      \rho_1(x)(y):=[x,y]_R=[Rx,y]+[x,Ry]=(\ad_{Rx}+\ad_x\circ
      R)(y), \quad \forall x,y\in \g.
      $$
      The second one comes from the left multiplication of the pre-Lie algebra structure:
      $$
      \rho_2(x)y:=x\cdot_R y=[Rx,y]=\ad_{Rx}y, \quad \forall x,y\in \g.
      $$
      The representation $\mrho$ in Corollary \ref{cor:repR} is apparently different from the above two representations.
    \end{rmk}
As a special case of Definition~\ref{de:opcoh}, we give
    \begin{defi}
Let $R$ be a Rota-Baxter operator  of weight zero  on a Lie algebra $\g$. Then the cohomology of the cochain complex $(\oplus_k\huaC^k(\g,\g),d_\mrho)$, where the Chevalley-Eilenberg coboundary operator $d_\mrho:\huaC^k(\g,\g)\longrightarrow \huaC^{k+1}(\g,\g)$ is given by Eq.~\eqref{eq:odiff},
      is called the \bf{cohomology of the Rota-Baxter operator $R$}.
    \end{defi}

    This cohomology can be used to control infinitesimal, formal and order $n$ deformations of Rota-Baxter operators of weight 0. We only give some details on infinitesimal deformations.

    \begin{defi}Let $R$ be a Rota-Baxter operator  of weight $0$ on a Lie algebra $\g$.
    \begin{itemize}
      \item[\rm(i)] Let $\huaR:\g\longrightarrow\g$ be a linear operator.  If for all $t\in \K$, $R_t:=R+t\huaR$ is a Rota-Baxter operator  of weight $0$ on  $\g$, we say that $\huaR$ generates a {\bf one-parameter infinitesimal deformation of $R$}.
           \item[\rm(ii)] Let $R_t^1:=R+t\huaR_1$ and $R_t^2:=R+t\huaR_2$ be two one-parameter infinitesimal deformations of $R$ generated by $\huaR_1$ and $\huaR_2$ respectively. They are said to be {\bf equivalent} if there exists an $x\in\g$ such that
$({\Id}_\g+t\ad_x,{\Id}_\g+t\ad_x)$ is a homomorphism from $R^2_t$ to $R^1_t$. In particular, a
deformation $R_t=R+t\huaR$ of   $R$ is said to
be {\bf trivial} if there exists an $x\in\g$ such that
$({\Id}_\g+t\ad_x,{\Id}_\g+t\ad_x)$ is a homomorphism  from $R_t$ to $R$.
    \end{itemize}
    \end{defi}

\begin{pro}
 Let $R$ be a Rota-Baxter operator  of weight $0$ on a Lie algebra $\g$. If $\huaR$ generates a one-parameter infinitesimal deformation of $R$, then $\huaR$ is a $1$-cocycle. Moreover, if two one-parameter infinitesimal deformations of $R$ generated by $\huaR_1$ and $\huaR_2$ are equivalent, then $\huaR_1$ and $\huaR_2$ are in the same cohomological class.
\end{pro}

\begin{defi}
   An element $x$ in a Rota-Baxter Lie algebra $(\g,R)$ is called a {\bf Nijenhuis element} if
      \begin{eqnarray}
        \label{eq:NijRB1}~[x,[Ry,x]+R[x,y]]&=&0,\quad \forall y\in\g,\\
       \label{eq:NijRB2}~[ [x,y],[x,z]]&=&0,\quad\forall y,z\in\g.
      \end{eqnarray}
    \end{defi}

    \begin{pro}
 Let $R$ be a Rota-Baxter operator  of weight $0$ on a Lie algebra $\g$.   If $\huaR$ generates a trivial one-parameter infinitesimal deformation of $R$, i.e.  there exists an $x\in\g$ such that
$({\Id}_\g+t\ad_x,{\Id}_\g+t\ad_x)$ is a homomorphism  from $R_t=R+t\huaR$ to $R$, then $x$ is a Nijenhuis element.

Conversely, for any Nijenhuis element $x\in\g$, $R_t:=R+t \huaR$ with $\huaR:=d_\mrho x$ is a trivial infinitesimal deformation of   $R$.
    \end{pro}

We next give some examples of Rota-Baxter operators  of weight
0 on low-dimensional Lie algebras where the Nijenhuis elements
can be explicitly determined.

    \begin{ex}
    \label{ex:2dim}
    {\rm
    Consider the unique $2$-dimensional non-abelian Lie algebra on $\mathbb C^2$. The  Lie bracket is given by
   $[e_1,e_2]=e_1$ for for a given basis $\{e_1,e_2\}$.
   For a matrix $\left(\begin{array}{cc}a_{11}&a_{12}\\
a_{21}&a_{22}\end{array}\right)$,
define
$$
   Re_1=a_{11}e_1+a_{21}e_2,\quad Re_2=a_{12}e_1+a_{22}e_2.
$$
Then $R$ is a Rota-Baxter operator of weight $0$ if and only if
   $$
   [Re_1,Re_2]=R([Re_1,e_2]+[e_1,Re_2]).
   $$
   By a straightforward computation, we conclude that
  $R$ is a Rota-Baxter operator of weight zero if and only if
  $$
  (a_{11}+a_{22}) a_{21}=0,\quad a_{11}a_{22}-a_{12}a_{21}=(a_{11}+a_{22})a_{11}.
  $$
So we have the following two cases to consider.

\noindent
(i) If $a_{21}=0$, then we deduce that $a_{11}=0$ and any $R=\left(\begin{array}{cc}0&a_{12}\\
   0&a_{22}\end{array}\right)$ is a Rota-Baxter operator of weight zero. In this case, $x=t_1e_1+t_2e_2$ is a Nijenhuis element if and only if
   $$
   t_2(a_{12}t_2-a_{22}t_1)=0.
   $$
Then for any $t_1\in\mathbb C,$ $t_1e_1$ is a Nijenhuis element for the Rota-Baxter Lie algebra $\Big(\mathbb C^2,[\cdot,\cdot],\left(\begin{array}{cc}0&a_{12}\\
   0&a_{22}\end{array}\right)\Big)$.

\noindent
(ii) If $a_{11}+a_{22}=0$, then $a_{11}a_{22}=a_{12}a_{21}$. In this case, $x=t_1e_1+t_2e_2$ is a Nijenhuis element if and only if
   $$
   t_1^2a_{21}-t_2^2a_{12}-t_1t_2(a_{11} -a_{22})=0.
   $$
   In particular, $e_1+e_2$ is a Nijenhuis element for the Rota-Baxter Lie algebra $\Big(\mathbb C^2,[\cdot,\cdot],\left(\begin{array}{cc}1&-1\\
1&-1\end{array}\right)\Big)$.
   }
    \end{ex}

\begin{ex}{\rm
The {\bf Heisenberg algebra} $H_3(\mathbb C)$ is the
three-dimensional  complex  Lie algebra with basis
elements $e_1, e_2$ and $e_3$ and with Lie brackets
\begin{eqnarray*}
[e_1,e_2]=e_3,\quad [e_1,e_3]=0,\quad  [e_2,e_3]=0.
\end{eqnarray*}
Consider a linear operator $R$ defined by $\left(\begin{array}{ccc}r_{11}&r_{12}&r_{13}\\
r_{21}&r_{22}&r_{23}\\
r_{31}&r_{32}&r_{33}\end{array}\right)$ with respect to the basis $\{e_1,e_2,e_3\}$.
It is straightforward to check that $R$ is a Rota-Baxter operator of weight $0$ if and only if
$$r_{13}=r_{23}=0,\quad
(r_{11}+r_{22})r_{33}=r_{11}r_{22}-r_{21}r_{12}.$$

Now let $R$ be a Rota-Baxter operator of weight 0. For all
$x,y,z\in H_3(\mathbb C)$, by the fact that $[H_3(\mathbb
C),H_3(\mathbb C)]\subset \mathbb Ce_3$, Eq. \eqref{eq:NijRB2}
holds automatically. Then by $r_{13}=r_{23}=0$, we deduce that
$[Ry,x]+R[x,y]\in \mathbb Ce_3$ for any $x,y\in H_3(\mathbb C)$,
 which implies that Eq. \eqref{eq:NijRB1} holds.
Thus, for any Rota-Baxter operator $R$, the set of Nijenhuis
elements of $(H_3(\mathbb C),[\cdot,\cdot],R)$ is the whole space
$H_3(\mathbb C)$.

Furthermore,
any Nijenhuis element $x=t_1e_1+t_2e_2+t_3e_3\in
H_3(\mathbb C), t_1,t_2,t_3\in\mathbb C$,  gives rise to a trivial deformation of the Rota-Baxter operator. Its generator
$\frkT$     is given by
$$
\frkT=d_\mrho x=\left(\begin{array}{ccc}0&0&0\\
0&0&0\\
(r_{11}-r_{33})t_2-r_{21}t_1&(r_{33}-r_{22})t_1+r_{12}t_2&0
\end{array}\right).
$$
}
\end{ex}

\section{Deformations of skew-symmetric $r$-matrices and triangular Lie bialgebras}\label{sec:rmat}

As to be recalled below, a skew-symmetric $r$-matrix corresponds to an $\huaO$-operator on a Lie
algebra with respect to the {\bf coadjoint representation}.
This suggests to define deformations of skew-symmetric $r$-matrices from their corresponding \oops. As it turns out, there is a natural way to define such deformations directly and these two approaches are mostly consistent, yet new information can be obtained by the comparison. We also obtain deformations of triangular Lie bialgebras from their connection with skew-symmetric $r$-matrices.

\subsection{Maurer-Cartan elements and deformations of skew-symmetric $r$-matrices}

Recall that the Lie bracket $[\cdot,\cdot]$ in a Lie algebra $\g$
naturally extends to a graded Lie bracket (known as the
Gerstenhaber bracket) on $\wedge^\bullet\g=\oplus_k\wedge^{k+1}\g$, for which we use the
same notation $[\cdot,\cdot]$. More precisely, we have
$$
[x_1\wedge\cdots \wedge x_p,y_1\wedge\cdots\wedge y_q]=(-1)^{i+j}[x_i,y_j]\wedge x_1\wedge\cdots\hat{x_i}\cdots \wedge x_p\wedge y_1\wedge\cdots\hat{y_j}\cdots\wedge y_q,
$$
for all $x_1,\cdots, x_p,y_1,\cdots, y_q\in\g.$ As already given in Definition~\ref{defi:O} (ii), an element $r\in\wedge^2\g$ is called a {\bf skew-symmetric $r$-matrix} if it satisfies the {\bf classical Yang-Baxter
equation}:
\begin{equation}
[r,r]=0.
\end{equation}
Thus we have the tautological statement that the Maurer-Cartan elements of this graded Lie algebra are simply the skew-symmetric $r$-matrices. Further by Proposition~\ref{pp:mce}, we have

\begin{thm}
  Let $\g$ be a Lie algebra and $r\in\wedge^2\g$ a skew-symmetric $r$-matrix.
  \begin{enumerate}
  \item
  The triple $(\oplus_k\wedge^{k+1}\g,[\cdot,\cdot],d_r)$ is a differential graded Lie algebra, in which elements in $\wedge^{k+1}\g$ are of degree $k$ and  $d_r:\wedge^k\g\longrightarrow\wedge^{k+1}\g$ is defined by
$  d_r=[r,\cdot].$
\item
Let $r'\in\wedge^2\g$. Then $r+r'$ is still a   skew-symmetric $r$-matrix if and only if $r'$ is a Maurer-Cartan element of the differential graded Lie algebra $(\oplus_k\wedge^{k+1}\g,[\cdot,\cdot],d_r)$.
\end{enumerate}
\end{thm}

An element $r\in\wedge^2\g$ naturally induces a linear map
$r^\sharp:\g^*\longrightarrow\g$ by
$$
\langle r^\sharp(\xi),\eta\rangle=r(\xi,\eta)=\langle r,
\xi\otimes\eta\rangle,\quad \forall
\xi,\eta\in\g^*.
$$

It is well known that $r$ satisfies the classical Yang-Baxter
equation if and only if $r^\sharp$ is an $\huaO$-operator on $\g$
with respect to the coadjoint representation~\cite{Ku}. The
associated pre-Lie algebra structure $\cdot_r$
on $\g^*$ is given by
$$
\xi\cdot_r\eta:=\ad^*_{r^\sharp(\xi)}\eta,\quad \forall
\xi,\eta\in\g^*.
$$
Its sub-adjacent Lie algebra structure
$[\cdot,\cdot]_r$ on $\g^*$ is given by
\begin{equation}\label{eq:bracketr}
[\xi,\eta]_r:=\ad^*_{r^\sharp(\xi)}\eta-\ad^*_{r^\sharp(\eta)}\xi,\quad
\forall \xi,\eta\in\g^*.
\end{equation}

By Theorem \ref{thm:deformation}, deformations of the corresponding \oop $r^\sharp$ are characterized by Maurer-Cartan elements of the differential graded Lie algebra $(\oplus_k\Hom(\wedge^k\g^*,\g),\Courant{\cdot,\cdot},d_{r^\sharp})$. We next establish a relationship between these two differential graded Lie algebras.

Recall that associated to the coadjoint representation, the graded Lie bracket $$\Courant{\cdot,\cdot}:\Hom(\wedge^n\g^*,\g)\times \Hom(\wedge^m\g^*,\g)\longrightarrow \Hom(\wedge^{n+m}\g^*,\g)$$
is given by Eq.~\eqref{o-bracket}
and the differential $d_{r^\sharp}$ is given by
$
d_{r^\sharp}=\Courant{r^\sharp,\cdot}.
$

For any $k\geq 0$, define $\Psi:\wedge^{k+1}\g\longrightarrow \Hom(\wedge^k\g^*,\g)$ by
\begin{equation}\label{eq:defipsi}
 \langle\Psi(P)(\xi_1,\cdots,\xi_k),\xi_{k+1}\rangle=\langle P,\xi_1\wedge\cdots\wedge\xi_k\wedge\xi_{k+1}\rangle,\quad \forall P\in\wedge^{k+1}\g, \xi_1,\cdots, \xi_{k+1}\in\g^*.
\end{equation}
In particular, for any $x\in\g$, $\Psi(x)=x$ and for any $r\in\wedge^2\g$, $\Psi(r)=r^\sharp$.
 The map $\Psi$ establishes a relationship between  the differential graded Lie algebra $(\oplus_k\wedge^{k+1}\g,[\cdot,\cdot],d_r)$ and the differential graded Lie algebra $(\oplus_k\Hom(\wedge^k\g^*,\g),\Courant{\cdot,\cdot},d_{r^\sharp})$ determined  by the \oop $r^\sharp.$

\begin{pro}\label{pro:dglamap1}
Let $\g$ be a Lie algebra and $r\in\wedge^2\g$ a skew-symmetric
$r$-matrix.
  For any $P\in\wedge^{p+1}\g$ and $Q\in\wedge^{q+1}\g$, we have
\begin{eqnarray}\label{eq:antimor}
 \Psi([P,Q])&=&(-1)^{pq}\Courant{\Psi(P),\Psi(Q)},\\
 \label{eq:antimor1}\Psi\circ d_r(P)&=&(-1)^pd_{r^\sharp}\circ\Psi(P).
\end{eqnarray}
\end{pro}
\begin{proof}
For all $x,y\in\g$, by the facts that $\Psi(x)=x$ and $[x,y]=\Courant{x,y}$, we have
$$
\Psi([x,y])=[x,y]=\Courant{\Psi(x),\Psi(y)}.
$$
Then the general case of Eq.~\eqref{eq:antimor} can be proved by an induction.

Further by Eq.~\eqref{eq:antimor}, we have
$\Psi\circ d_r(P)=\Psi([r,P])=(-1)^p\Courant{r^\sharp,\Psi(P)}=(-1)^pd_{r^\sharp}\circ\Psi(P).
$
\end{proof}

\begin{cor} Let $\g$ be a Lie algebra and $r\in\wedge^2\g$ a skew-symmetric
$r$-matrix.
 If $r'\in\wedge^2\g$ is a Maurer-Cartan element of the differential graded Lie algebra $(\oplus_k\wedge^{k+1}\g,[\cdot,\cdot],d_r)$, then $r'^\sharp$ is a Maurer-Cartan element of the differential graded Lie algebra $(\oplus_k\Hom(\wedge^k\g^*,\g),\Courant{\cdot,\cdot},d_{r^\sharp})$.
\end{cor}

We now recall that a Lie bialgebra is a vector space $\g$ equipped
with a Lie algebra structure
$[\cdot,\cdot]:\wedge^2\g\longrightarrow\g$ and a Lie coalgebra
structure $\delta:\g\longrightarrow\wedge^2\g$ such that $\delta$
is a 1-cocycle on $\g$ with coefficients in $\wedge^2\g$ via
the tensor product of adjoint representations. Note that a Lie
coalgebra structure on $\g$ is equivalent to a Lie algebra
structure on $\g^*$ when $\g$ is finite-dimensional. A Lie
bialgebra homomorphism  between two Lie
bialgebras $(\g,[\cdot,\cdot],\delta_\g)$ and $(\h,[\cdot,\cdot],\delta_\h)$  is  a Lie
algebra homomorphism $\phi:\g\longrightarrow\h$ such that
$$(\phi\otimes \phi)\circ \delta_\g=\delta_\h\circ\phi.$$
In particular, a Lie bialgebra isomorphism is a  Lie algebra
isomorphism $\phi:\g\longrightarrow\h$ such that
${\phi}^*:\h^*\longrightarrow\g^*$ is also a Lie algebra
isomorphism.

Let $r$ be a skew-symmetric $r$-matrix. Define
$\delta:\g\longrightarrow\wedge^2\g$ by
\begin{equation}
\delta(x)=[x,r],\quad\forall x\in\g. \mlabel{eq:delta}
\end{equation}
Then $(\g,[\cdot,\cdot],\delta)$ is a Lie bialgebra, which is
called a {\bf triangular Lie bialgebra}. Note that such a
$\delta$ defines a Lie algebra structure on $\g^*$ which is
exactly the one given by Eq.~\eqref{eq:bracketr}.

\begin{rmk}
  By Eq. \eqref{eq:antimor}, we can recover a very useful formula in the theory of Lie bialgebras and Poisson geometry:
  $$
 \half [r,r](\xi,\eta,\cdot)=[r^\sharp(\xi),r^\sharp(\eta)]-r^\sharp([\xi,\eta]_r),\quad\forall r\in\wedge^2\g, \xi,\eta\in\g.
  $$
It follows from
\begin{eqnarray*}
  \half [r,r](\xi,\eta,\cdot)=\Psi(\half[r,r])(\xi,\eta)=\half\Courant{r^\sharp,r^\sharp}(\xi,\eta) =[r^\sharp(\xi),r^\sharp(\eta)]-r^\sharp([\xi,\eta]_r), \forall r\in\wedge^2\g, \xi,\eta\in\g.
\end{eqnarray*}
\end{rmk}

\subsection{A controlling cohomology of deformations of skew-symmetric $r$-matrices}

Now we establish an analogue of the Andr\'e-Quillen cohomology for
skew-symmetric $r$-matrices to control deformations
of skew-symmetric $r$-matrices. Let
$r\in\wedge^2\g$ be a skew-symmetric $r$-matrix. Then $(\g^*,[\cdot,\cdot]_r)$ is a Lie algebra, where the Lie
bracket $[\cdot,\cdot]_r$ is given by Eq.~\eqref{eq:bracketr}. Let
$(\oplus_k\Hom(\wedge^k\g^*,\K),\dM)$ be the cochain complex
associated to the trivial representation, where the coboundary operator 
$\dM:\Hom(\wedge^k\g^*,\K)\longrightarrow
\Hom(\wedge^{k+1}\g^*,\K)$ is given by
\begin{eqnarray*}
  \dM f(\xi_1,\cdots,\xi_{k+1})=\sum_{i<j}(-1)^{i+j}f([\xi_i,\xi_j]_r,\xi_1,\cdots,\hat{\xi_i},\cdots,\hat{\xi_j},\cdots,\xi_{k+1}),\  \forall f\in \Hom(\wedge^k\g^*,\K), \xi_i\in \g^*.
\end{eqnarray*}
Denote by $H^k(\g^*)$ the $k$-th cohomology group,
called {\bf the $k$-th cohomology group of the skew-symmetric
$r$-matrix $r$}. We will identify $\Hom(\wedge^{k}\g^*,\K)$ with
$ \wedge^k\g $ in the sequel.

\begin{pro}\label{pro:danddr}
Let $\g$ be a Lie algebra and $r\in\wedge^2\g$ a skew-symmetric $r$-matrix. Then we have
\begin{equation}
  \dM f=d_rf:=[r,f],\quad\forall f\in\Hom(\wedge^{k}\g^*,\K)=\wedge^k\g.
\end{equation}
\end{pro}
\begin{proof}
For all $f=x\in\g$ and $\xi_1\wedge\xi_2\in\wedge^2\g^*$, we have
\begin{eqnarray*}
\langle d_rx,\xi_1\wedge\xi_2\rangle&=&\langle [r,x],\xi_1\wedge\xi_2\rangle=\langle r,\ad^*_x(\xi_1\wedge\xi_2)\rangle=\langle r,\ad^*_x\xi_1\wedge\xi_2+\xi_1\wedge\ad^*_x\xi_2\rangle\\
&=&-\langle r^\sharp(\xi_2),\ad^*_x\xi_1\rangle+\langle r^\sharp(\xi_1),\ad^*_x\xi_2\rangle=\langle [x,r^\sharp(\xi_2)],\xi_1\rangle-\langle [x,r^\sharp(\xi_1)],\xi_2\rangle\\
&=&\langle x,\ad^*_{r^\sharp(\xi_2)}\xi_1\rangle-\langle x,\ad^*_{r^\sharp(\xi_1)}\xi_2\rangle=-\langle x,[\xi_1,\xi_2]_r\rangle=\langle \dM x,\xi_1\wedge\xi_2\rangle.
\end{eqnarray*}
Thus,  $d_rx=\dM x$. Arguing by induction, assume that the conclusion holds for $f=P\in\wedge^n\g$, that is,
\begin{equation}
\label{dr=d}
\sum_{1\le i<j\le n+1} 
(-1)^{i+j}\langle
P,[\xi_i,\xi_j]_r\wedge\xi_1\wedge\cdots\hat{\xi_i}\cdots\hat{\xi_j}\cdots\wedge\xi_{n+1}\rangle
=\langle [r,P],\xi_1\wedge\cdots\wedge \xi_{n+1}\rangle,  
\end{equation}
for all $\xi_1,\cdots,\xi_{n+1}\in \g^*$. Then for $f=x_1\wedge P$ and $\xi_1\wedge\cdots\wedge \xi_{n+2}\in\wedge^{n+2}\g^*$, we have
\begin{eqnarray*}
&&\langle d_r(x_1\wedge P)-\dM(x_1\wedge P) ,\xi_1\wedge\cdots\wedge \xi_{n+2}\rangle\\
&=&\langle [r,x_1]\wedge P+(-1)^{1\cdot(2-1)}x_1\wedge[r,P] ,\xi_1\wedge\cdots\wedge \xi_{n+2}\rangle\\
&&-\sum_{1\le i<j\le n+2}(-1)^{i+j}\langle x_1\wedge P,[\xi_i,\xi_j]_r\wedge\xi_1\wedge\cdots\hat{\xi_i}\cdots\hat{\xi_j}\cdots\wedge\xi_{n+2}\rangle\\
&=&\sum_{1\le i<j\le n+2}(-1)^{1+2+i+j}\langle [r,x_1],\xi_i\wedge\xi_j\rangle\langle P,\xi_2\wedge\cdots\hat{\xi_i}\cdots\hat{\xi_j}\cdots\wedge\xi_{n+2}\rangle\\
&&-\sum_{1\le i\le n+2}(-1)^{1+i}\langle x_1,\xi_i\rangle\langle[r,P] ,\xi_1\wedge\cdots\hat{\xi_i}\cdots\wedge \xi_{n+2}\rangle\\
&&-\sum_{1\le i<j\le n+2}(-1)^{i+j}\langle x_1,[\xi_i,\xi_j]_r\rangle\langle P,\xi_2\wedge\cdots\hat{\xi_i}\cdots\hat{\xi_j}\cdots\wedge\xi_{n+2}\rangle\\
&&-\sum_{1\le i<j\le n+2}(-1)^{i+j}\sum_{1\le s\le i-1}(-1)^{s}\langle x_1,\xi_s\rangle\langle P,[\xi_i,\xi_j]_r\wedge\xi_1\wedge\cdots\hat{\xi_s}\cdots\hat{\xi_i}\cdots\hat{\xi_j}\cdots\wedge\xi_{n+2}\rangle\\
&&-\sum_{1\le i<j\le n+2}(-1)^{i+j}\sum_{i+1\le s\le j-1}(-1)^{1+s}\langle x_1,\xi_s\rangle\langle P,[\xi_i,\xi_j]_r\wedge\xi_1\wedge\cdots\hat{\xi_i}\cdots\hat{\xi_s}\cdots\hat{\xi_j}\cdots\wedge\xi_{n+2}\rangle\\
&&-\sum_{1\le i<j\le n+2}(-1)^{i+j}\sum_{j+1\le s\le n+2}(-1)^{s}\langle x_1,\xi_s\rangle\langle P,[\xi_i,\xi_j]_r\wedge\xi_1\wedge\cdots\hat{\xi_i}\cdots\hat{\xi_j}\cdots\hat{\xi_s}\cdots\wedge\xi_{n+2}\rangle.
\end{eqnarray*}
By $\langle[r,x],\xi_1\wedge\xi_2\rangle=-\langle x,[\xi_1,\xi_2]_r\rangle$, the sum of the first and third terms is zero. Next, by  Eq.\eqref{dr=d}, the second term is
\begin{eqnarray*}
&&\sum_{1\le i\le n+2}(-1)^{i}\langle x_1,\xi_i\rangle\langle[r,P] ,\xi_1\wedge\cdots\hat{\xi_i}\cdots\wedge \xi_{n+2}\rangle\\
&=&\sum_{1\le i\le n+2}(-1)^{i}\sum_{1\le s<t\le i-1}(-1)^{s+t}\langle x_1,\xi_i\rangle\langle P ,[\xi_s,\xi_t]_r\wedge\xi_1\wedge\cdots\hat{\xi_s}\cdots\hat{\xi_t}\cdots\hat{\xi_i}\cdots\wedge \xi_{n+2}\rangle\\
&&+\sum_{1\le i\le n+2}(-1)^{i}\sum_{1\le s<i<t\le n+2}(-1)^{s+t-1}\langle x_1,\xi_i\rangle\langle P ,[\xi_s,\xi_t]_r\wedge\xi_1\wedge\cdots\hat{\xi_s}\cdots\hat{\xi_i}\cdots\hat{\xi_t}\cdots\wedge \xi_{n+2}\rangle\\
&&+\sum_{1\le i\le n+2}(-1)^{i}\sum_{i+1\le s<<t\le n+2}(-1)^{s-1+t-1}\langle x_1,\xi_i\rangle\langle P,[\xi_s,\xi_t]_r\wedge\xi_1\wedge\cdots\hat{\xi_i}\cdots\hat{\xi_s}\cdots\hat{\xi_t}\cdots\wedge \xi_{n+2}\rangle,
\end{eqnarray*}
which implies that $d_r(x_1\wedge P)=\dM(x_1\wedge P)$. This completes the induction.
\end{proof}

There is a close relationship between the cohomology group $H^k(\g^*)$ and the cohomology group
$\huaH^{k-1}(\g^*,\g)$ of the \oop $r^\sharp.$ Let us recall the
latter from Section 3. By Lemma \ref{lem:rep}, we have
\begin{cor}
   Let $r\in\wedge^2\g$ be a skew-symmetric $r$-matrix. Let $\g$ be a Lie algebra and $r\in\wedge^2\g$ a skew-symmetric
$r$-matrix. Then
   $$
   \mrho:\g^*\longrightarrow \gl(\g), \quad
   \mrho(\xi)(x):=[r^\sharp(\xi),x]+r^\sharp\ad_x^*\xi,\quad \forall \xi\in\g^*, x\in \g,
   $$
   is a representation of the Lie algebra $(\g^*,[\cdot,\cdot]_r)$ on the vector space $\g$.
\end{cor}

\begin{rmk}
  The representation $\mrho$ given above is exactly the coadjoint representation of the Lie algebra $(\g^*,[\cdot,\cdot]_r)$ on the vector space $\g$. More precisely, let $\add:\g^*\longrightarrow\gl(\g^*)$ be the adjoint representation of the Lie algebra $(\g^*,[\cdot,\cdot]_r)$, then $\mrho=\add^*$. It follows from
  \begin{eqnarray*}
    \langle\add_\xi^*x,\eta\rangle&=&-\langle x,[\xi,\eta]_r\rangle=-\langle x,\ad^*_{r^\sharp\xi}\eta-\ad^*_{r^\sharp\eta}\xi\rangle\\
    &=&\langle [r^\sharp\xi,x],\eta\rangle+\langle [x,r^\sharp\eta],\xi\rangle=\langle [r^\sharp\xi,x],\eta\rangle+\langle r^\sharp\ad_x^*\xi,\eta\rangle\\
    &=&\langle\mrho(\xi)(x),\eta\rangle,\quad\forall \xi,\eta\in\g^*, x\in\g.
  \end{eqnarray*}
\end{rmk}
The Chevalley-Eilenberg coboundary operator $d_\mrho: \Hom(\wedge^k\g^*,\g)\longrightarrow  \Hom(\wedge^{k+1}\g^*,\g)$ of the representation $\mrho$ is given by
              \begin{eqnarray*}
                && d_\mrho f(\xi_1,\cdots,\xi_{k+1})\\
                &:=&\sum_{i=1}^{k+1}(-1)^{i+1}[r^\sharp\xi_i,f(\xi_1,\cdots,\hat{\xi_i},\cdots, \xi_{k+1})]+
                \sum_{i=1}^{k+1}(-1)^{i+1}r^\sharp\ad^*_{f(\xi_1,\cdots,\hat{\xi_i},\cdots, \xi_{k+1})}\xi_i\\&&+\sum_{1\le i<j\le k+1}(-1)^{i+j}f([\xi_i,\xi_j]_r,\xi_1,\cdots,\hat{\xi_i},\cdots,\hat{\xi_j},\cdots,
                \xi_{k+1}), \forall f\in
                \Hom(\wedge^k\g^*,\g),\xi_1,\cdots,\xi_{k+1}\in
                \g^*.
               \end{eqnarray*}

 \begin{thm} With the notations as above, the map $\Psi$ defined by Eq.~\eqref{eq:defipsi} is a cochain map between cochain complexes $(\oplus_k\wedge^k\g,\dM)$ and $(\oplus_k\Hom(\wedge^k\g^*,\g),d_\mrho)$.
    Consequently, $\Psi$ induces a map $\Psi_*$ between the corresponding cohomology groups.
 \end{thm}
 \begin{proof}
   By Propositions \ref{pro:dglamap1}, \ref{pro:danddr} and \ref{pro:danddT}, for all $P\in  \wedge^{k+1}\g$, we have
   \begin{eqnarray*}
     \Psi (\dM P)=\Psi([r,p])=(-1)^k[r^\sharp,\Psi(P)]=d_\mrho (\Psi(P)),
   \end{eqnarray*}
as needed.
\end{proof}

 \begin{cor}
Let $\g$ be a Lie algebra and $r\in\wedge^2\g$ a skew-symmetric $r$-matrix.
Then for the corresponding $\mathcal O$-operator $r^\sharp$
associated to the coadjoint representation, we have
$$d_\mrho(x)=[r,x]^\sharp,\;\;\forall x\in \g.$$
\end{cor}

\subsection{Weak homomorphisms between skew-symmetric $r$-matrices and Lie bialgebras}
We now apply the connection between deformations of \oops and
those of skew-symmetric $r$-matrices to study weak
homomorphisms between skew-symmetric $r$-matrices.

\begin{defi}\label{defi:iso}
Let $\g$ be a Lie algebra and $r_1,\;r_2$  two skew-symmetric
$r$-matrices. A {\bf weak homomorphism} from   $r_2$ to $r_1$ consists of a Lie algebra homomorphism $\phi:\g\rightarrow\g$ and
a linear map $\varphi:\g\rightarrow \g$ satisfying
\begin{eqnarray}(\varphi\otimes {\rm Id}_\g)(r_1)&=&({\rm
Id}_g\otimes \phi)(r_2);\label{eq:req1}\\
\varphi[\phi(x),y]&=&[x,\varphi(y)],\;\;\forall x,y\in
\g.\label{eq:req2} \end{eqnarray}
If $\phi$ and
$\varphi$ are also linear isomorphisms, then $(\phi,\varphi)$ is called
a {\bf weak isomorphism} from $r_2$ to $r_1$.
\end{defi}

\begin{pro} Let $\g$ be a Lie algebra and $r_1,\;r_2$  two skew-symmetric
$r$-matrices. Then $(\phi,
\varphi)$ is a weak homomorphism (weak isomorphism) from $r_2$  to $r_1$ if
and only if  $(\phi,\varphi^*)$ is a homomorphism (isomorphism) from $r_2^\sharp$ to
$r_1^\sharp$ as $\mathcal O$-operators  on $\g$ with respect to the
coadjoint representation.
\end{pro}

\begin{proof}
Let $r_1=\sum_ja_i\otimes b_i$ and
$r_2=\sum_jx_j\otimes y_j$. Then
$$r_1^\sharp(\xi)=\sum_i\langle \xi,a_i\rangle
b_i,\;\;r_2^\sharp(\xi)=\sum_j\langle \xi,x_j\rangle
y_j,\;\;\forall \xi\in \g^*. $$  Hence for any $\xi,\eta\in\g^*$,
we have
\begin{eqnarray*}
\langle (\varphi\otimes {\rm Id}_\g)(r_1),\xi\otimes \eta\rangle
&=&\sum_i\langle \varphi (a_i),\xi\rangle\langle
b_i,\eta\rangle=\sum_i\langle a_i,\varphi^*(\xi)\rangle\langle
b_i,\eta\rangle= \langle r_1^\sharp\circ
\varphi^*(\xi),\eta\rangle,\\
\langle ({\rm Id}_\g)\otimes \phi) (r_2),\xi\otimes \eta\rangle
&=&\sum_j\langle x_j,\xi\rangle\langle
\phi(y_i),\eta\rangle=\langle \phi(\sum_j\langle
x_j,\xi\rangle y_i),\eta\rangle= \langle\phi\circ
r_2^\sharp(\xi),\eta\rangle.
\end{eqnarray*}
Therefore $\phi\circ r_2^\sharp=r_1^\sharp\circ \varphi^*$ holds if and only if
Eq.~(\ref{eq:req1}) holds.

It is straightforward to check that Eq.~(\ref{eq:req2}) holds if
and only if $\varphi^*{\rm ad}^*_x\xi={\rm
ad}^*_{\phi(x)}\varphi^*(\xi)$ for all $x\in \g, \xi\in \g^*.$
Hence the conclusion holds.
\end{proof}

Recall from~\cite{CP} that two skew-symmetric $r$-matrices $r_1$
and $r_2$ are said to be {\bf equivalent} if there is a Lie
algebra isomorphism $\phi:\g\longrightarrow\g$ such that
  \begin{equation}
    (\phi\otimes \phi) (r_2)=r_1.
  \mlabel{eq:defi-eqr1}
  \end{equation}
There one can also find the notion of an equivalence of $r$-matrices up to a scalar.

\begin{cor}\label{co:r-eq} Let $\g$ be a Lie algebra and $r_1,\;r_2$
skew-symmetric $r$-matrices. Then $r_1$ is equivalent to $r_2$
if and only if there exists a Lie algebra isomorphism
$\phi:\g\longrightarrow\g$ such that $(\phi,{\phi^{-1}})$  is a weak
isomorphism from $r_2$ to $r_1$.
\end{cor}

\begin{proof}
If $\phi:\g\longrightarrow\g$ is an equivalence from $r_2$ to
$r_1$, then it is straightforward to check that
$(\phi,{\phi^{-1}})$  satisfies Eqs.~(\ref{eq:req1}) and
(\ref{eq:req2}). Conversely, Eqs.~(\ref{eq:req1}) implies
$(\phi\otimes \phi) (r_2)=r_1$.
\end{proof}

\begin{rmk} Therefore the notion of equivalence of two skew-symmetric
$r$-matrices $r_1$ and $r_2$ given in \cite{CP}  is not the
same as the notion of weak isomorphism between $r_1$ and $r_2$ in the
sense of  Definition~\ref{defi:iso} which is induced from the
notion of isomorphism between the corresponding $\mathcal
O$-operators with respect to the coadjoint representation. In
fact, in general, two weak isomorphic skew-symmetric
$r$-matrices in the sense of Definition~\ref{defi:iso} might not
be equivalent.
\end{rmk}

\begin{defi} Let
$(\g,[\cdot,\cdot],\delta_1)$ and $(\g,[\cdot,\cdot],\delta_2)$ be two Lie bialgebras,
$(\g^*,[\;,\;]^*_1)$ and $(\g^*,[\;,\;]^*_2)$  the corresponding
Lie algebra structures on $\g^*$ respectively.  A {\bf weak homomorphism}
from $(\g,[\cdot,\cdot],\delta_2)$ to $(\g,[\cdot,\cdot],\delta_1)$ consists of a Lie algebra homomorphism $\phi:\g\rightarrow \g$
and a linear map $\varphi:\g\rightarrow \g $ such that
$\varphi^*:(\g^*,[\cdot,\cdot]^*_2)\rightarrow (\g^*,[\cdot,\cdot]_1^*)$ is
a Lie algebra homomorphism (that is, $\varphi$ is a Lie
coalgebra homomorphism) and
\begin{equation}\varphi[\phi(x),y]=[x,\varphi(y)],\;\;\forall x,y\in
\g.\label{eq:req3} \end{equation} If in addition, both $\phi$ and
$\varphi$ are linear isomorphisms, then $(\phi,\varphi)$ is called
a {\bf weak isomorphism} from $(\g,[\cdot,\cdot],\delta_2)$ to $(\g,[\cdot,\cdot],\delta_1)$.
\end{defi}

\begin{rmk}
Note that the above notions of weak homomorphisms and weak
isomorphisms  are only available for the two Lie bialgebras with
the same Lie algebra $\g$, not for arbitrary two Lie bialgebras.
See also Remark~\ref{rmk:co}.
\end{rmk}

Straightforward from the definitions, we have
\begin{pro}\label{pro:Leq}
  Let $(\g,[\cdot,\cdot],\delta_1)$ and $(\g,[\cdot,\cdot],\delta_2)$
be two Lie bialgebras. Then
$(\g,[\cdot,\cdot],\delta_1)$ is isomorphic to
$(\g,[\cdot,\cdot],\delta_2)$ if and only if there exists a Lie
algebra isomorphism $\phi:\g\rightarrow \g$ such that
$(\phi,\phi^{-1})$ is a weak isomorphism from
$(\g,[\cdot,\cdot],\delta_2)$ to $(\g,[\cdot,\cdot],\delta_1)$.
\end{pro}

\begin{pro}\label{pro:r-L}
Let $\g$ be a Lie algebra and $r_1,\;r_2$ skew-symmetric
$r$-matrices. Let $(\g,[\cdot,\cdot],\delta_1)$   and $(\g,[\cdot,\cdot],\delta_2)$ be the
induced triangular Lie bialgebras respectively, that is,
$$
  \delta_1(x)=[x,r_1],\quad \delta_2(x)=[x,r_2],\quad \forall x\in\g.
  $$
If $(\phi,\varphi)$  is a weak homomorphism (weak isomorphism) from $r_2$  to $r_1$, or equivalently, if $(\phi,\varphi)$  is  a homomorphism (isomorphism) from the \oop  $r_2^\sharp$   to $r_1^\sharp$, then $(\phi,\varphi)$ is a weak homomorphism (weak isomorphism) from the Lie bialgebra $(\g,[\cdot,\cdot],\delta_2)$ to $(\g,[\cdot,\cdot],\delta_1)$.
\end{pro}

\begin{proof}
We only need to prove the case of homomorphisms. Let
$(\phi,\varphi)$ be a weak homomorphism from $r_2$ to $r_1$ as
skew-symmetric $r$-matrices.
 Let $\xi,\eta\in
\g^*, x\in \g$. Then we have
\begin{eqnarray*}
\langle \varphi^*[\xi,\eta]_{r_2},x\rangle&=& \langle
-[r_2^\sharp(\xi),\varphi(x)],\eta\rangle+\langle
[r_2^\sharp(\eta),\varphi(x)],\xi\rangle\\
&=& -\langle \varphi[\phi r_2^\sharp(\xi),x], \eta\rangle+
\langle \varphi[\phi r_2^\sharp(\eta),x], \xi\rangle\\
&=&-\langle \varphi[r_1^\sharp \varphi^*(\xi),x],\eta\rangle
+\langle \varphi[r_1^\sharp \varphi^*(\eta),x],\xi\rangle\\
&=& -\langle [r_1^\sharp
\varphi^*(\xi),x],\varphi^*(\eta)\rangle +\langle [r_1^\sharp
\varphi^*(\eta),x],\varphi^*(\xi)\rangle\\
 &=&\langle[\varphi^*(\xi),\varphi^*(\eta)]_{r_1},x\rangle.
\end{eqnarray*}
Hence $\varphi^*$ is a Lie algebra homomorphism from
$(\g^*,[\cdot,\cdot]_{r_2})$ to $(\g^*,[\cdot,\cdot]_{r_1})$. Note that
Eq.~(\ref{eq:req3}) is exactly Eq.~(\ref{eq:req2}). So the
conclusion holds.
\end{proof}

Combining Proposition~\ref{pro:r-L}, Proposition~\ref{pro:Leq} ad
Corollary~\ref{co:r-eq}, we obtain

\begin{cor} Let $\g$ be a Lie algebra and $r_1$, $r_2$ skew-symmetric $r$-matrices. Let $(\g,[\cdot,\cdot],\delta_1)$   and $(\g,[\cdot,\cdot],\delta_2)$ be the induced triangular Lie bialgebras. If $r_1$ is equivalent to $r_2$, then
$(\g,\delta_1)$ is isomorphic to $(\g,\delta_2)$ as Lie
bialgebras.
\end{cor}

\subsection{Infinitesimal deformations of skew-symmetric $r$-matrices and triangular Lie bialgebras}

\begin{defi}Let $\g$ be a Lie algebra and $r\in\wedge^2\g$  a
skew-symmetric $r$-matrix.   If $r+t\kappa$ is a
skew-symmetric $r$-matrix for any $t$, then we say that $\kappa$
generates a {\bf one-parameter infinitesimal deformation} of $r$.
\end{defi}

\begin{defi}
Two one parameter infinitesimal deformations
  $r_t^1=r+t\kappa_1$ and $r_t^2=r+t\kappa_2$ of $r$ are called
{\bf equivalent} if there exists $x\in\g$ such that $({\rm
Id}_\g+t{\rm ad}_x,{\rm Id}_\g-t{\rm ad}_x)$ is a weak homomorphism
from $r_t^2$ to $r_t^1$. In particular, a one-parameter
infinitesimal deformation $r_t=r+t\kappa$ is called {\bf
trivial} if there exists $x\in\g$ such that $({\rm Id}_\g+t{\rm
ad}_x,{\rm Id}_\g-t{\rm ad}_x)$ is a weak homomorphism from $r_t$ to
$r$.
\end{defi}

\begin{pro} Let $\g$ be a Lie algebra and $r\in\wedge^2\g$  a
skew-symmetric $r$-matrix.
\begin{itemize}
  \item[\rm(i)]If $\kappa$
generates a   one-parameter infinitesimal deformation  of $r$, then $\kappa$ is a $2$-cocycle.

  \item[\rm(ii)]If two one-parameter infinitesimal deformations of skew-symmetric
$r$-matrices $r_t^1=r+t\kappa_1$ and $r_t^2=r+t\kappa_2$ are
equivalent, then $\kappa_1$ and $\kappa_2$ are in
the same cohomology class of
$H^2(\g^*)$.
\end{itemize}

\end{pro}

\begin{proof}
 (i) If $r+t\kappa$ is a skew-symmetric $r$-matrix for any $t$, then we have
 \begin{eqnarray*}
   [r+t\kappa,r+t\kappa]=[r,r]+2t[r,\kappa]+t^2[\kappa,\kappa]=0,
 \end{eqnarray*}
 which implies
$
 [r,\kappa]=0,
$
and hence $\dM \kappa=0$ by Proposition~\ref{pro:danddr}.

 (ii) Assume that $({\rm
Id}_\g+t{\rm ad}_x,{\rm Id}_\g-t{\rm ad}_x)$ is a weak homomorphism
from $r_t^2$ to $r_t^1$.  Then there exists an $x\in\g$ such that
\begin{eqnarray*}
  \Id_\g\otimes (\Id_\g+t\ad_x)(r+t\kappa_2)=(\Id-t\ad_x)\otimes \Id_\g (r+t\kappa_1),
\end{eqnarray*}
which implies
$$
t(\kappa_2-\kappa_1+[x,r])+t^2((\Id_\g\otimes\ad_x)\kappa_2+(\ad_x\otimes\Id_\g)\kappa_1)=0.
$$
Therefore, we have
$
\kappa_2-\kappa_1=[r,x],
$
which implies $\kappa_2-\kappa_1=\dM x$. This completes the proof.
\end{proof}

From the fact that the dual map of $\ad^*_x:\g^*\longrightarrow\g^*$ is $-\ad_x:\g\longrightarrow\g$ for all $x\in\g$, we have

\begin{pro}\label{pro:rrr}
Let $\g$ be a Lie algebra and $r\in\wedge^2\g$ a skew-symmetric $r$-matrix.
\begin{itemize}
  \item[\rm(i)]  $\kappa\in\wedge^2\g$ generates a one-parameter infinitesimal
deformation of $r$ if and only if $\kappa^\sharp$ generates a
one-parameter infinitesimal deformation of $r^\sharp$ as $\mathcal
O$-operators.
 \item[\rm(ii)]Two one-parameter infinitesimal deformations
  $r_t^1=r+t\kappa_1$ and $r_t^2=r+t\kappa_2$ of $r$ are
 equivalent if and only if ${r_t^1}^\sharp$ and ${r_t^2}^\sharp$ are equivalent as \oops.
\end{itemize}
\end{pro}

Now we consider trivial deformations of a skew-symmetric $r$-matrix which lead to the definition of Nijenhuis elements associated to a skew-symmetric $r$-matrix. Let $r_t=r+t\kappa$ be a trivial deformation of a skew-symmetric $r$-matrix $r$. Then there exists an $x\in\g$ such that $({\rm Id}_\g+t{\rm
ad}_x,{\rm Id}_\g-t{\rm ad}_x)$ is a weak homomorphism from $r_t$ to
$r$. First by the fact that $\Id_\g+t\ad_x$ is a Lie algebra endomorphism, we get
\begin{eqnarray*}
(\Id_\g+t\ad_x)[y,z]&=&[(\Id_\g+t\ad_x)(y),(\Id_\g+t\ad_x)(z)]\\
&=&[y,z]+t([[x,y],z]+[y,[x,z]])+t^2[[x,y],[x,z]], \quad
\forall y,z\in \g,
\end{eqnarray*}
which implies
 \begin{equation}\label{eq:Nijr1}
   [[x,y],[x,z]]=0, \quad \forall y,z \in\g.
 \end{equation}
Then by Eq.~\eqref{eq:req1}, we get
 \begin{equation}\label{eq:Nijr2}
  ({\rm Id}_\g\otimes {\rm ad}_x)({\rm Id}_\g\otimes {\rm ad}_x+{\rm ad}_x\otimes
{\rm Id}_\g)(r)=0.
 \end{equation}
By Eq.~\eqref{eq:req2}, we get
 \begin{equation}\label{eq:Nijr3}
  [x,[[x,y],z]]=0,\quad \forall y,z\in \g.
 \end{equation}

\begin{defi} Let $\g$ be a Lie algebra.
An element   $x\in\g$ is called a {\bf Nijenhuis element}
associated to a skew-symmetric $r$-matrix $r\in\wedge^2\g$ if $x$ satisfies Eqs. \eqref{eq:Nijr1}, \eqref{eq:Nijr2} and \eqref{eq:Nijr3}.
\end{defi}

Denote by $\Nij(r)$ the set of
Nijenhuis elements associated to a skew-symmetric $r$-matrix $r$.

\begin{rmk} Obviously,
$x\in\g$ is   a   Nijenhuis element
associated to a skew-symmetric $r$-matrix $r$ if and only if $x$ is a Nijenhuis element associated to the \oop $r^\sharp$ with respect to the coadjoint representation, that is,
$
\Nij(r)=\Nij(r^\sharp).
$
\end{rmk}

\begin{ex}{\rm
  Consider the unique $2$-dimensional non-abelian Lie algebra in Example~\ref{ex:2dim}. It is obvious that for any $a\in\mathbb C$, $ae_1\wedge e_2$ is a skew-symmetric $r$-matrix, and for any $b\in\mathbb C$, $be_1$ is a Nijenhuis element associated to $ae_1\wedge e_2$.
   }
\end{ex}

From the above discussion, we have seen that a trivial deformation
of a skew-symmetric $r$-matrix gives rise to a
Nijenhuis element. Conversely, we have

\begin{pro}
Let $\g$ be a Lie algebra and $r\in\wedge^2\g$ a skew-symmetric $r$-matrix.
Then for any  $x\in \Nij(r)$,
$r_t=r+t[r,x]$
is a trivial  one-parameter infinitesimal deformation of $r$.
\end{pro}

\begin{proof}
By Theorem \ref{thm:trivial}, since $x$ is also a Nijenhuis element associated to the \oop $r^\sharp$, $r_t^\sharp=r^\sharp+t[r,x]^\sharp=r^\sharp+td_\mrho(x)$ is a trivial deformation of $r^\sharp$. Then by Proposition \ref{pro:rrr}, $r_t=r+t[r,x]$ is a trivial  one-parameter
infinitesimal deformation of $r$.
\end{proof}

In the sequel, we consider one-parameter infinitesimal deformations of a
Lie bialgebra.

\begin{defi} Let $(\g,[\cdot,\cdot],\delta)$ be a Lie
bialgebra and $\gamma:\g\longrightarrow\wedge^2\g$ a linear map.
 If
$\delta+t\gamma$ defines a Lie bialgebra structure
 on the Lie algebra $\g$ for any $t$, then we say that $\gamma$
generates a {\bf one-parameter infinitesimal deformation} of the
Lie bialgebra $(\g,[\cdot,\cdot],\delta)$.
\end{defi}

The following conclusion is obvious.

\begin{pro}\label{pro:deforLiebir}
 Let $(\g,[\cdot,\cdot],\delta)$ be a triangular Lie bialgebra induced by
a skew-symmetric $r$-matrix $r$ through Eq.~\eqref{eq:delta}. If $\kappa$ generates a {one-parameter infinitesimal
deformation} of $r$, then $\gamma$ defined from $\kappa$ by
Eq.~\eqref{eq:delta} generates a one-parameter infinitesimal
deformation of the Lie bialgebra $(\g,[\cdot,\cdot],\delta)$.
\end{pro}

\begin{defi} Let $(\g,[\cdot,\cdot],\delta)$ be a Lie bialgebra.
Two one-parameter infinitesimal deformations
$\delta_t^1=\delta+t\gamma_1$ and $\delta_t^2=\delta+t\gamma_2$
are said to be {\bf equivalent} if there exists an $x\in\g$ such
that $({\rm Id}_\g+t{\rm ad}_x,{\rm Id}_\g-t{\rm ad}_x)$ is a weak
homomorphism from $(\g,[\cdot,\cdot],\delta_t^2)$ to $(\g,[\cdot,\cdot],\delta_t^1)$. In particular, a
one-parameter infinitesimal deformation $\delta_t=\delta+t\gamma$
is said to be {\bf trivial} if there exists an $x\in\g$ such that
$({\rm Id}_\g+t{\rm ad}_x,{\rm Id}_\g-t{\rm ad}_x)$ is a weak
homomorphism from $(\g,[\cdot,\cdot],\delta_t)$ to $(\g,[\cdot,\cdot],\delta)$.
\end{defi}

\begin{pro}
 Let $(\g,[\cdot,\cdot],\delta)$ be a triangular Lie bialgebra induced by
a skew-symmetric $r$-matrix $r$ through Eq.~\eqref{eq:delta}.
Assume that $r_t^1=r+t\kappa_1$ and $r_t^2=r+t\kappa_2$ are two
one-parameter infinitesimal deformations of $r$, and
$\delta_t^1=\delta+t\gamma_1$ and $\delta_t^2=\delta+t\gamma_2$
are the corresponding one-parameter infinitesimal deformations of
$(\g,[\cdot,\cdot],\delta)$ given in Proposition
\ref{pro:deforLiebir} respectively. Then $\delta_t^1$
and $\delta_t^2$ are equivalent if and only if $r_t^1=r+t\kappa_1$
and $r_t^2=r+t\kappa_2$ are equivalent.
\end{pro}
\begin{proof}
For this one checks that $(\Id_\g-t\ad_x)^*=\Id_{\g^*}+t\ad_x^*$ is a Lie algebra morphism from $(\g^*,[\cdot,\cdot]^*_2)$ to $(\g^*,[\cdot,\cdot]^*_1)$ if and only if
$
  \Id_\g\otimes (\Id_\g+t\ad_x)(r+t\kappa_2)=(\Id-t\ad_x)\otimes \Id_\g (r+t\kappa_1).
$
\end{proof}

\begin{cor}
Let $(\g,[\cdot,\cdot],\delta)$ be a triangular Lie bialgebra induced by a
skew-symmetric $r$-matrix $r$ through Eq.~\eqref{eq:delta}. Then
for any  $x\in \Nij(r)$,
$$\delta_t(y)=\delta(y)+t[y,[r,x]],\;\;\forall y\in \g$$ is a trivial  one-parameter
infinitesimal deformation of $(\g,[\cdot,\cdot],\delta)$.
\end{cor}

\subsection{Further discussions}

By a similar approach as above, we can also study the formal
deformations of skew-symmetric $r$-matrices in terms of the formal
deformations of the corresponding $\mathcal O$-operators given in
Section 5.

The relations among deformations of various objects obtained above can be summarized in the following diagram:
\vspace{-.1cm}
{\footnotesize
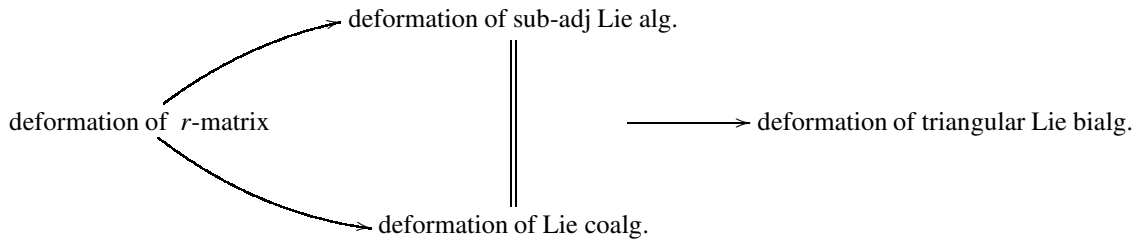
 \[
 \xymatrix{
&\mbox{deformation of sub-adj Lie alg.}\ar@{=}[dd]& \\
\mbox{deformation of }~ r\mbox{-matrix}\ar@/^1.4pc/[ur]\ar@/_1.4pc/[dr] &\qquad\qquad\qquad\qquad\ar[r]& \mbox{deformation of triangular Lie bialg.}\\
&\mbox{deformation of Lie coalg.}& }
\]
}

We end the paper with some observations on related topics for
future consideration.

\begin{rmk}\label{rmk:co}
In the above deformations of a Lie bialgebra we
deform the Lie coalgebra structure but leave the
underlying Lie algebra structure intact. This is consistent with
the overall approach of this paper that we have deformed the \oops
and pre-Lie algebras while fixing the underlying Lie algebra.
Results in this paper pave the way to consider a deformation
theory where the Lie algebra is also deformed.
\end{rmk}

\begin{rmk}
It is natural to consider the quantum enveloping algebra
structures corresponding to the deformations of the above
triangular Lie bialgebras. It is known that for every Lie bialgebra,
there is a corresponding quantum enveloping algebra~\cite{EK}. So the
quantum enveloping algebras corresponding to the deformations of
the triangular Lie bialgebras in this subsection should be certain
``deformed" structures also. This problem could be better
understood with more explicit examples. On the other hand, there
have been some interest in ``multiparameter quantization". In
fact, some approaches on this subject (for example~\cite{Re}) have
already involved certain ``deformed" structures of $r$-matrices.
It is an interesting problem to study the relationships between
these quantum structures.
\end{rmk}

\noindent{\bf Acknowledgements. } This research is supported by
NSFC (11471139,  11425104, 11771190) and NSF of Jilin
Province (20170101050JC).

\end{document}